\documentclass[12pt]{amsart}
\usepackage[all]{xy}
\usepackage[parfill]{parskip}
\usepackage{hyperref}
\usepackage{pinlabel}
\usepackage{enumerate}
\usepackage{verbatim}
\usepackage{amsmath, amscd, graphics, latexsym, times}

\usepackage{subcaption}
\oddsidemargin .25in \theoremstyle{plain}
\newtheorem{dummy}{anything}[section]
\newtheorem{theorem}[dummy]{Theorem}

\newtheorem{lemma}[dummy]{Lemma}

\newtheorem{proposition}[dummy]{Proposition}

\newtheorem{corollary}[dummy]{Corollary}

\theoremstyle{definition}
\newtheorem{definition}[dummy]{Definition}

\newtheorem{example}[dummy]{Example}

\newtheorem{remark}[dummy]{Remark}

\theoremstyle{remark}

\newcommand{\A}{\mathcal{A}}

\newcommand{\del}{\partial}

\newcommand{\C}{\mathbb{C}}

\newcommand{\OB}{\mathcal{OB}}

\def\cp{\hbox{${\mathbb C} {\mathbb P}^2$}}

\def\cpb{\hbox{$\overline{{\mathbb C} {\mathbb P}^2}$}}

\newcommand{\Si}{\Sigma}
\def\a{\alpha}
\def\b{\beta}
\def\d{\delta}
\def\g{\gamma}
\def\G{\Gamma}
\def\l{\lambda}
\def\e{\varepsilon}

\begin{document}

\title{Trisections of 4-manifolds via Lefschetz fibrations}

\author{Nickolas A. Castro and Burak Ozbagci}

\address{Department of Mathematics, UC Davis, Davis, CA 95616}
\email{ncastro@math.ucdavis.edu} 

\address{Department of Mathematics, UCLA, Los Angeles, CA 90095 and Department of Mathematics, Ko\c{c} University, 34450,  Istanbul, Turkey}
\email{bozbagci@ku.edu.tr}

\subjclass[2000]{}

\thanks{The second author was partially supported by a research grant of the Scientific and Technological Research Council of Turkey. }


\begin{abstract} We develop a technique for  gluing relative trisection diagrams of $4$-manifolds with nonempty connected boundary to obtain  trisection diagrams for  closed $4$-manifolds.  As an application,  we describe a  trisection of any closed $4$-manifold which admits a Lefschetz fibration over $S^2$ equipped with a section of square $-1$, by an explicit diagram determined by the vanishing cycles of the Lefschetz fibration. In particular, we obtain a trisection diagram for some simply connected minimal complex surface of general type. As a consequence, we obtain  explicit trisection diagrams for  a pair of  closed $4$-manifolds which are homeomorphic but not diffeomorphic. Moreover, we describe a trisection for any oriented $S^2$-bundle over any closed surface and in particular  we draw the  corresponding diagrams for $T^2 \times S^2$ and $T^2 \tilde{\times} S^2$ using our gluing technique.   Furthermore, we provide an  alternate proof of a recent result of Gay and Kirby which says that every closed $4$-manifold admits a trisection. The key feature of our proof is that Cerf theory takes a back seat to contact geometry.   
\end{abstract}
 
\maketitle

\section{Introduction}\label{sec: intro}

Recently, Gay and Kirby \cite{gk} proved that every smooth, closed, oriented, connected $4$-manifold $X$ admits a trisection, meaning that  for every $X$, there exist non-negative integers $g \geq k$ such that $X$ is diffeomorphic to the union $X_1 \cup X_2 \cup X_3$ of three copies of the $4$-dimensional $1$-handlebody $ X_i \cong \natural^k (S^1 \times  B^3) $, intersecting pairwise in $3$-dimensional handlebodies, with triple intersection a closed, oriented, connected $2$-dimensional surface $\Si_g$ of genus $g$.  Such a decomposition of $X$ is called a $(g,k)$-trisection or simply a genus $g$ trisection, since $k$ is determined by $g$ using the fact that $\chi (X)=2+g-3k$, where $\chi (X)$ denotes the Euler characteristic of $X$. 

Moreover, they showed that the trisection data can be encoded  as a $4$-tuple $(\Si_g, \a, \b, \g)$, which is called a $(g,k)$-trisection diagram,  such that each triple $(\Si_g, \a, \b)$, $(\Si_g, \b, \g)$, and  $(\Si_g, \g, \a)$ is a genus $g$ Heegaard diagram for  ${\#}^k (S^1 \times S^2)$.   Furthermore, they proved that trisection of $X$ (and its diagram)  is unique up to a natural stabilization operation. 

On the other hand, various flavors of Lefschetz fibrations have been studied extensively in the last two decades to understand the topology of smooth $4$-manifolds. Suppose that a closed $4$-manifold $X$ admits a Lefschetz fibration over $S^2$, whose regular fiber is a smooth, closed, oriented, connected  surface $\Si_p$ of genus $p$. The fibration induces a handle decomposition of $X$, where the essential data can be encoded by a finite set of ordered simple closed curves (the vanishing cycles) on a surface diffeomorphic to $\Si_p$. The only condition imposed on the set curves is that the product of right-handed Dehn twists along these curves is isotopic to the identity diffeomorphism of $\Si_p$.
 
In addition, every $4$-manifold $W$ with nonempty boundary has a  relative trisection and under favorable circumstances $W$ also admits an \emph{achiral} Lefschetz fibration over $B^2$ with bounded fibers. The common feature shared by these structures is that each induces a natural open book on $\del W$. To exploit this feature in the present paper, we develop a technique to obtain  trisection diagrams for closed $4$-manifolds by gluing relative trisection diagrams of $4$-manifolds with nonempty connected boundary. The precise result is stated in Proposition~\ref{prop:gluediagram},  which is too technical to include in the introduction. Nevertheless, our gluing technique has several applications --- one of which is the following result.  
 
{\bf Theorem 3.7.} {\em Suppose that $X$ is a  smooth, closed, oriented, connected $4$-manifold  which admits a genus $p$ Lefschetz fibration over $S^2$ with $n$ singular fibers, equipped with a section of square $-1$. Then, an explicit  $(2p+n+2, 2p)$-trisection of $X$ can be described  by a  corresponding trisection diagram,  which is  determined by the vanishing cycles of the Lefschetz fibration.     Moreover, if $\widetilde{X}$ denotes the $4$-manifold obtained from $X$ by blowing down the section of square $-1$, then we also obtain a $(2p+n+1, 2p)$-trisection of $\widetilde{X}$ along with a corresponding diagram.  }

In particular, Theorem~\ref{thm: main} provides a description of a $(46, 4)$-trisection diagram of the Horikawa surface $H'(1)$ (see \cite[page 269]{gs} for its definition and properties),    a simply connected  complex surface of general type which admits a genus $2$ Lefschetz fibration over $S^2$ with $40$ singular fibers, equipped with a section of square $-1$.  This section is the unique sphere in $H'(1)$ with self-intersection $-1$ so that  by blowing it down,  we obtain a trisection diagram for the simply connected  {\em minimal} complex surface $\widetilde{H'(1)}$  of general type.  

To the best of our knowledge, none of the existing methods in the literature can be effectively utilized to obtain explicit trisection diagrams for complex surfaces of general type. For example,  Gay and Kirby describe  trisections of $S^4$, $\cp$, $\cpb$, closed $4$-manifolds admitting  locally trivial fibrations over $S^1$ or $S^2$ (including of course $S^1 \times S^3$, $S^2 \times S^2$ and $S^2 \tilde{\times} S^2$) and  arbitrary  connected sums of these in  \cite{gk}. 

Note that, by Freedman's celebrated theorem, the Horikawa surface $H'(1)$ is homeomorphic to $5 \cp \# 29 \cpb$ (and also to the elliptic surface $E(3)$), since it is simply connected, nonspin and its Euler characteristic is $36$, while its signature is $-24$.  On the other hand, since $H'(1)$ is a simply connected complex surface (hence K\"{a}hler) with $b_2^+ (H'(1)) >1$, it has non-vanishing Seiberg-Witten invariants,  while   $5 \cp \# 29 \cpb$ has vanishing Seiberg-Witten invariants which follows from the fact that $5 \cp \# 29 \cpb = \cp \# (4 \cp \# 29 \cpb)$. Hence, we conclude that  $H'(1)$ is certainly not diffeomorphic to $5 \cp \# 29 \cpb$. 

As a consequence, we obtain explicit $(46,4)$-trisection diagrams for a pair of  closed $4$-manifolds, the Horikawa surface $H'(1)$ and  $5 \cp \# 29 \cpb$, which are homeomorphic but not diffeomorphic. Note that $5 \cp \# 29 \cpb$ has a natural $(34,0)$-trisection diagram (obtained by the connected sum of the standard $(1,0)$-trisection diagrams of $\cp$ and $\cpb$), which can be  
stabilized four times to yield a  $(46,4)$-trisection diagram.

More generally,  Theorem~\ref{thm: main}   can be applied to a large class of $4$-manifolds. A fundamental result of Donaldson \cite{don} says that every closed \emph{symplectic}  $4$-manifold admits a Lefschetz pencil over $\C \mathbb{P}^1 \cong S^2$. By blowing up its base locus the Lefschetz pencil can be turned into a Lefschetz fibration over $S^2$, so that each exceptional sphere becomes a (symplectic)  section of square $-1$. Conversely,  any $4$-manifold $X$ satisfying the hypothesis of Theorem~\ref{thm: main} must carry a symplectic structure where the section of square $-1$ can be assumed to be symplectically embedded.  Therefore, $X$ is necessarily  a nonminimal symplectic $4$-manifold.   


In \cite[Theorem 3]{g}, Gay describes a trisection for any closed $4$-manifold $X$ admitting a Lefschetz pencil, although he does not formulate the trisection of $X$ in terms of the vanishing cycles of the pencil (see \cite[Remark 9]{g}). He also points out that his technique does not extend to cover the case  of Lefschetz fibrations on closed $4$-manifolds \cite[Remark 8]{g}.  

We would like to point out that Theorem~\ref{thm: main}   holds true for any \emph{achiral} Lefschetz fibration  $\pi_X:  X \to S^2$  equipped with a section of square $-1$. In this case, $X$ is not necessarily symplectic. We opted to state our result only for Lefschetz fibrations   to emphasize their connection to symplectic geometry.

Next  we turn our attention to another natural application of our  gluing technique where we find trisections of  doubles of $4$-manifolds with nonempty connected boundary.  It is well-known (see, for example, \cite[Example 4.6.5]{gs}) that there are  two oriented $S^2$-bundles over  a closed, oriented,  connected surface $\Si_h$ of genus $h$: the trivial  bundle $\Si_h \times S^2$ and the twisted bundle  $\Si_h \tilde{\times} S^2$. The former is the double of any $B^2$-bundle over $\Si_h$ with even Euler number, while the latter is  the double of any $B^2$-bundle over $\Si_h$ with odd Euler number. We obtain  trisections of  these $S^2$-bundles by doubling the relative trisections of the appropriate  $B^2$-bundles. In particular, we draw the  corresponding $(7,3)$-trisection diagram for $T^2 \times S^2$ and the $(4,2)$-trisection diagram for $T^2 \tilde{\times} S^2$ using our gluing technique. 

For any  $h \geq 1$, the twisted bundle $\Si_h \tilde{\times} S^2$ is not covered by the examples in \cite{gk}, while our trisection for $\Si_h \times S^2$  has smaller genus compared to that of given in \cite{gk}.  We discuss the case of  oriented $S^2$-bundles over nonorientable surfaces in Section~\ref{subsec:nonorie}. 

Finally, we provide  a simple alternate proof of the following result due to Gay and Kirby. 

{\bf Theorem 5.1.}  {\em Every smooth, closed, oriented, connected $4$-manifold admits a trisection.}

Our proof is genuinely different from the two original proofs due to Gay and Kirby \cite{gk}, one with Morse $2$-functions and one with ordinary Morse functions,  since not only contact geometry plays a crucial role in our proof, but we also employ a technique for gluing relative trisections.

After the completion of our work, we learned that Baykur and Saeki \cite{bs}  gave yet another proof of Theorem~\ref{thm: existence}, setting up a correspondence between broken Lefschetz fibrations and trisections on $4$-manifolds, using a method which is very different from ours. In particular, they prove the  existence of a $(2p+k+2, 2k)$-trisection  on a 4-manifold $X$ which admits a genus $p$ Lefschetz fibration over $S^2$ with $k$ Lefschetz singularities --- generalizing the first assertion in our Theorem~\ref{thm: main}, but without providing the corresponding explicit diagram for the trisection.  They also give examples of trisections (without diagrams) on a pair of closed $4$-manifolds (different from ours)  which are homeomorphic but not diffeomorphic.  In addition, for any $h \geq 0$, they give small genus trisections (again without the diagrams) for $\Si_h \times S^2$.

{\bf Conventions:} All $4$-manifolds are assumed to be smooth, compact, oriented and connected throughout the paper. The corners which appear in gluing manifolds are smoothed in a canonical way.

\section{Gluing relative trisections} 

We first review some  basic results about trisections and their diagrams (cf. \cite{gk}).  Let $Y^+_{g,k} \cup Y^-_{g,k}$ denote the standard genus $g$ Heegaard splitting of $\#^kS^1\times S^2 $ obtained by stabilizing the standard genus $k$ Heegaard splitting $g-k$ times.

\begin{definition} \label{def:tris}
A $(g,k)$-trisection   of a closed $4$-manifold $X$ is a decomposition $X = X_1 \cup X_2 \cup X_3$ such that for each $i=1,2,3$, 
	\begin{enumerate}[i)]
		\item there is a diffeomorphism $\varphi_i: X_i \rightarrow \natural^k S^1 \times B^3,$ and 
		\item taking indices mod $3$, $\varphi_i(X_i \cap X_{i+1}) = Y^+_{g,k}$ and $\varphi_i(X_i \cap X_{i-1}) = Y^-_{g,k}$. 
	\end{enumerate} \end{definition}
It follows that  $X_1 \cap X_2 \cap X_3$ is a closed surface of genus $g$.  Also note that $g$ and $k$ determine each other, since  
the Euler characteristic $\chi(X)$ is  equal to $2+g-3k$, which can be easily derived by gluing $X_1$ and $X_2$ first and then gluing $X_3$. 

Suppose that each of $\eta$ and $\zeta$ is a collection of $m$ disjoint simple closed curves on some compact surface $\Sigma$. We say that two  such triples $(\Si, \eta, \zeta)$ and $(\Si', \eta', \zeta')$ are \emph{diffeomorphism and handleslide equivalent} if there exists a diffeomorphism $h: \Si \to \Si'$ such that $h(\eta)$ is related to $\eta'$ by a sequence of handleslides and $h(\zeta)$ is related to $\zeta'$ by a sequence of handleslides.  

\begin{definition} \label{def:trisd}
A $(g,k)$-trisection diagram is an ordered $4$-tuple $(\Sigma, \alpha, \beta, \gamma)$ such that
\begin{enumerate}[i)]
	\item $\Sigma$ is a closed genus $g$ surface,
	\item each of $\alpha, \beta$, and  $\gamma$ is a  non-separating collection of $g$ disjoint, simple closed curves on $\Si_g$, 
	\item each triple $(\Sigma, \alpha, \beta)$, $(\Sigma, \beta, \gamma),$ and $(\Sigma, \alpha, \gamma)$ is diffeomorphism and handleslide equivalent  to the standard genus $g$ Heegaard diagram of $\#^k S^1\times S^2$ depicted in Figure~\ref{fig:standardpairs}. 
\end{enumerate}
\end{definition}

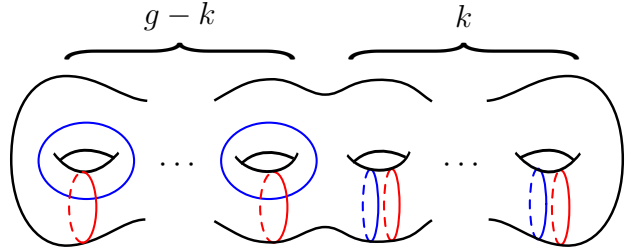
\begin{figure}[ht]
\labellist
		\pinlabel $\cdots$ at 64 31
		\pinlabel {\rotatebox{90}{\resizebox{9pt}{2.3cm}{$\}$}}} at 64 74
		\pinlabel $g-k$ at 64 87
		
		\pinlabel $\cdots$ at 171 31
		\pinlabel {\rotatebox{90}{\resizebox{9pt}{2.3cm}{$\}$}}} at 171 74
		\pinlabel $k$ at 171 87
	\endlabellist	
			\includegraphics[scale=1]{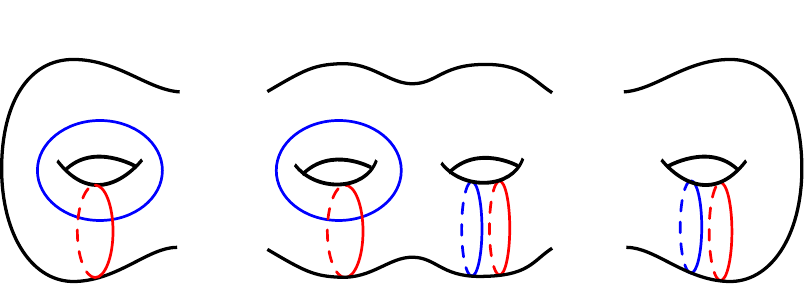}
	\caption{The standard genus $g$ Heegaard diagram of $\#^kS^1\times S^2 $ obtained by stabilizing the standard genus $k$ Heegaard diagram  $g-k$ times. ~\label{fig:standardpairs} }
\end{figure}

According to Gay and Kirby \cite{gk}, every closed $4$-manifold admits a trisection, which in turn, can be encoded by a diagram. Conversely,  every trisection diagram determines a trisected closed $4$-manifold,  uniquely up to diffeomorphism.

Next we recall the analogous definitions of \emph{relative}  trisections and their diagrams  for $4$-manifolds with nonempty connected boundary (cf. \cite{gk, cgp}). 

Suppose that $W$ is a $4$-manifold with nonempty connected boundary $\del W$. We would like to find a decomposition  $W=W_1 \cup W_2 \cup W_3$, such that each $W_i$ is diffeomorphic to $\natural^k  S^1 \times B^3$ for some fixed $k$.  Since $\partial W \neq \emptyset$, it would be natural to require that part of each $\del W_i$ contribute to $\del W$.    Hence, we need a particular decomposition of $\del(\natural^k  S^1 \times B^3 )=  \#^k S^1\times S^2$ to specify a submanifold  of $\del W_i$ to be embedded in $\partial W$. With this goal in mind, we proceed as follows to develop the language we will use throughout the paper.

Suppose that  $g,k,p,b$ are non-negative integers  satisfying $b >0$ and $$g+p+b-1 \geq k \geq 2p+b-1.$$  Let $Z_k = \natural^k S^1\times B^3$ and $Y_k=\partial Z_k = \#^k S^1\times S^2$ for $k \geq 1$,  and  $Z_0 = B^4$, $Y_0=S^3$.

We denote by $P$ a fixed genus $p$ surface with $b$ boundary components. Let $$D=\{re^{i\theta} \in \mathbb{C} \; | 
\;  r \in [0,1] \; \mbox{and}  \; -\pi/3\leq \theta \leq \pi/3\}$$ be a third of the unit disk in the complex plane whose boundary is decomposed as $\partial D = \partial^-D \cup \partial^0 D\cup  \partial^+D$, where 
\begin{align*}
	\partial^-D& = \{re^{i\theta} \in \partial D \; | \; \theta=-\pi/3\}\\
	\partial^0D& = \{e^{i\theta} \in \partial D\}\\
	\partial^+D& =\{re^{i\theta} \in \partial D \; | \; \theta=\pi/3\}
\end{align*} The somewhat unusual choice of the disk $D$ will be justified by the construction below.  We set  $U = P \times D$, which is indeed diffeomorphic to $\natural^{2p+b-1}S^1\times B^3$. Then,   $\partial U$ inherits a decomposition $\partial U = \partial^-U \cup \partial^0U \cup \partial ^+U$, where  $\partial^0U=(P \times \partial^0D) \cup (\partial P \times D)$ and $\partial^{\pm}U = P \times \partial^{\pm}D$.

Let $\partial(S^1\times B^3) = H_1 \cup H_2$ be the standard genus one Heegaard splitting of $S^1\times S^2$. For any $n>0$, let $V_n = \natural^n(S^1\times B^3)$, where the boundary connected sum is taken in neighborhoods of the Heegaard surfaces, inducing the standard genus $n$ Heegaard splitting of $\partial V_n =\#^n S^1\times S^2 = \partial^-V_n \cup \partial^+V_n$. Stabilizing this Heegaard splitting $s$ times we obtain a genus $n+s$ Heegaard splitting of $\partial V_n = \partial^-_s V_n \cup \partial^+_s V_n$.

We set $s=g-k+p+b-1$ and  $n=k-2p-b+1$.  Note that we can identify $Z_k = U\natural V_n$,  where the boundary connected sum again takes place along the neighborhoods of points in the Heegaard surfaces. We now have a decomposition of $\partial Z_k$ as follows: 
\begin{align*}
	\partial Z_k = Y_k = Y^+_{g,k;p,b} \cup Y^0_{g,k;p,b} \cup Y^-_{g,k;p,b},
\end{align*}
where $Y^\pm_{g,k;p,b} = \partial^\pm U \natural \partial^{\pm}_s V_n$ and $Y^0_{g,k;p,b}= \partial^0 U.$

\begin{definition} \label{def: rela}
A $(g,k;p,b)$-relative trisection of a $4$-manifold $W$ with non-empty connected boundary is a decomposition $W= W_1 \cup W_2 \cup W_3$ such that for each $i=1,2,3$,
	\begin{enumerate}[i)]
		\item there is a diffeomorphism $\varphi_i: W_i \rightarrow Z_k = \natural^k S^1 \times B^3,$ and 
		\item taking indices mod $3$, $\varphi_i(W_i \cap W_{i+1}) = Y^+_{g,k;p,b}$ and $\varphi_i(W_i \cap W_{i-1}) = Y^-_{g,k;p,b}.$
	\end{enumerate}

\end{definition}
As a consequence, $W_1 \cap W_2 \cap W_3$ is diffeomorphic to  $Y^-_{g,k;p,b} \cap Y^+_{g,k;p,b}$, which is a genus $g$ surface with $b$ boundary components.  Note that the Euler characteristic $\chi(W)$ is equal to  $g-3k+3p+2b-1$, which can be calculated directly from the definition of a relative trisection. We also give alternate method to calculate $\chi(W)$  in Corollary~\ref{cor: euler}.

According to \cite{gk}, every  $4$-manifold $W$ with nonempty connected boundary admits a trisection. Moreover, there is a natural open book induced on $\partial W$, whose page is diffeomorphic to $P$, which is an  essential ingredient in our definition of  $Y^{\pm}_{g,k;p,b}$.

Informally, the contribution of each $\del W_i$ to $\del W$ is  one third of an open book.   This is because the part of each $\del W_i$ that contributes to $\del W$ is diffeomorphic to $$Y^0_{g,k;p,b} = \partial^0U=(P \times \partial^0D) \cup (\partial P \times D), $$ where  $P \times \partial^0D$ is one third of the truncated pages, while  $\partial P \times D$ is one third of the neighborhood of the binding.   In other words,  not only  we trisect the $4$-manifold $W$, but we also trisect its boundary $\del W$. Conversely,  if an open book is fixed on $\partial W$, then $W$ admits a trisection whose induced open book coincides with the given one.

\begin{definition}
A $(g,k;p,b)$-relative trisection diagram is an  ordered $4$-tuple $(\Sigma, \alpha, \beta, \gamma)$ such that 
\begin{enumerate}[i)]
	\item $\Sigma$ is a genus $g$ surface with $b$ boundary components,
	\item each of $\alpha, \beta$ and $\gamma$ is a collection of $g-p$ disjoint, essential, simple closed curves,
	\item each triple $(\Sigma, \alpha, \beta), (\Sigma, \beta, \gamma)$, and $(\Sigma, \alpha, \gamma)$ is diffeomorphism and handleslide equivalent to the diagram depicted in  Figure~\ref{fig:StandardPosition}.
\end{enumerate}

\end{definition}

\begin{figure}[ht] 
\vspace*{12pt}
	\labellist
		\pinlabel \rotatebox{-90}{\resizebox{8pt}{.85in}{$\{$}} at 64 84
		\pinlabel \resizebox{.9in}{!}{$g+p+b-1-k$} at 64 97
		\pinlabel $\cdots$ at 65 42
		
		\pinlabel \rotatebox{-90}{\resizebox{8pt}{.85in}{$\{$}} at 169 84
		\pinlabel \resizebox{.86in}{!}{$k-2p-b+1$} at 169 97
		\pinlabel $\cdots$ at 170 42
		
		\pinlabel \rotatebox{-90}{\resizebox{8pt}{.85in}{$\{$}} at 274 84
		\pinlabel $p$ at 274 97
		\pinlabel $\cdots$ at 275 42
		
		\pinlabel \resizebox{7pt}{.8in}{$\}$} at 355 43
		\pinlabel $b$ at 365 43
		\pinlabel \rotatebox{90}{$\cdots$} at 345 43
	\endlabellist
\includegraphics[scale=.85]{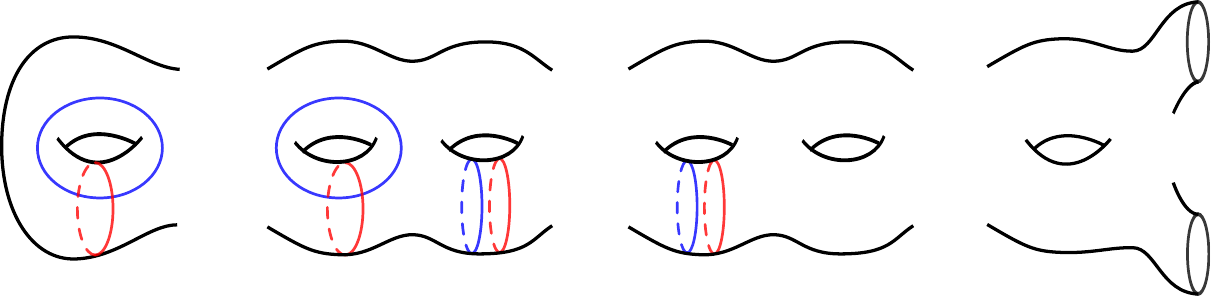}
	\caption{ ~\label{fig:StandardPosition}}
	\end{figure}

It was shown in \cite{cgp} that every relative trisection diagram determines uniquely,  up to diffeomorphism, (i)  a relatively trisected $4$-manifold $W$ with nonempty connected boundary and (ii) the open book on $\del W$ induced by the trisection. Moreover, the page and the monodromy of the open book on  $\partial W$ is determined completely by the relative trisection diagram by an explicit algorithm, which we spell out below. 

Suppose that $(\Si, \a, \b ,\g)$  is a  $(g,k;p,b)$-relative trisection diagram, which represents a relative trisection of a  $4$-manifold $W$ with nonempty connected boundary.  The page of the induced open book $\OB$ on $\del W$ is given by $\Si_\a$, which is the genus $p$ surface with $b$ boundary components   obtained from $\Si$ by performing surgery along the $\a$ curves. This means that to obtain $\Si_\a$,  we cut open $\Si$ along each $\a$ curve and glue in disks to cap off the resulting boundaries. 

Now that we have a fixed identification of the page of $\OB$ as $\Si_\a$, we use Alexander's trick  to describe the monodromy $\mu : \Si_\a \to \Si_\a$ of $\OB$.   Namely,  we cut $\Si_\a$ into a single disk via two distinct ordered collections of $2p+b-1$ arcs, so that for each arc in one collection there is an arc in the other collection with the same endpoints. As a result,  we get a self-diffeomorphism of $S^1$ that takes one collection of arcs to the other respecting the ordering of the arcs, and equals to the identity otherwise. This diffeomorphism uniquely extends to a  self-diffeomorphism of  the disk, up to isotopy.  Therefore,  we get a self-diffeomorphism $\mu$  of $\Si_\a$ fixing $\del \Si_\a$ pointwise, which is uniquely determined up to isotopy. Next, we provide some more details (see \cite[Theorem 5]{cgp}) about how to obtain the aforementioned collection of arcs.

Let $\A_\alpha$ be any ordered collection of disjoint, properly embedded $2p+b-1$ arcs in $\Sigma$ disjoint from $\a$, such that the image of $\A_\a$ in $\Sigma_\a$ cuts $\Sigma_\alpha$  into a disk. We choose a collection of arcs $\A_\b$,  and a collection of simple closed curves $\b'$ disjoint from  $\A_\b$ in $\Si$ such that $(\a, \A_\b)$ is handleslide equivalent to $(\a, \A_\a)$, and $\b'$ is handleslide equivalent to $\b$. This means that $\A_\b$ arcs are obtained by sliding $\A_\a$ arcs over $\a$ curves, and $\b'$ is obtained by sliding  $\b$  curves over $\b$ curves. Next we choose a collection of arcs $\A_\g$, and a collection of simple closed curves $\g'$ disjoint from  $\A_\g$ in $\Si$ such that $(\b', \A_\g)$ is handleslide equivalent to $(\b', \A_\b)$, and $\g'$ is handleslide equivalent to $\g$. This means that $\A_\g$ arcs are obtained by sliding $\A_\b$ arcs over $\b'$ curves,  and $\g'$ is obtained by sliding  $\g$  curves over $\g$ curves.
Finally,  we choose a collection of arcs $\A$, and a collection of simple closed curves $\a'$ disjoint from  $\A$ in $\Si$ such that $(\g', \A)$ is handleslide equivalent to $(\g', \A_\g)$, and $\a'$ is handleslide equivalent to $\a$. This means that $\A$ arcs are obtained by sliding $\A_\g$ arcs over $\g'$ curves,  and $\a'$ is obtained by sliding  $\a$  curves over $\a$ curves. It follows that $(\a', \A)$  is handleslide equivalent to $(\a, \A_*)$ for some collection of arcs $\A_*$ disjoint from $\a$ in $\Si$.

\begin{definition} \label{def: cutarc} We call the triple $(\A_\a, \A_\b, \A_\g)$ a cut system of arcs associated to the diagram  $(\Si, \a, \b ,\g)$.  \end{definition} 

Now we have two ordered collections of $2p+b-1$ arcs $\A_\a$ and $\A_*$ in $\Si \setminus \a$, such that each of their images in $\Si_\a$ cuts $\Si_\a$ into a disk. Then, as we explained above, there is a unique diffeomorphism  $\mu:\Sigma_\alpha \rightarrow \Sigma_\alpha$, up to isotopy,  which fixes $\partial \Si_\a$ pointwise such that  $\mu(\A_\alpha)= \A_*$.  

\begin{remark}
It is shown in \cite{cgp} that, up to isotopy,  the monodromy of the resulting open book is independent of the choices in the above algorithm. 
\end{remark}

Next,  we give a very simple version of the general gluing theorem \cite{c} for relatively trisected $4$-manifolds. Here we present a different proof --- where we use the definition of a relative trisection as given in \cite{cgp} instead of \cite{gk} --- for the case of a single boundary component.

\begin{lemma} \label{lem:gluem}
Suppose that  $W$ and $W'$ are  $4$-manifolds such that $\partial W$ and $\partial W'$ are both nonempty and connected.  Let $W=W_1 \cup W_2 \cup W_3$ and $W'=W_1' \cup W_2'\cup W_3'$  be $(g, k; p, b)$- and $(g', k';  p', b')$-relative trisections with  induced open books $\OB$ and $\OB'$ on $\partial W$ and $\partial W'$, respectively. If $f: \partial W \rightarrow \partial W'$ is an orientation-reversing diffeomorphism which takes  $\OB$ to $\OB'$ (and hence $p'=p$ and $b'=b$), then the relative trisections on $W$ and $W'$ can be glued together to yield a $(G,K)$-trisection of the closed $4$-manifold $X= W \cup_f W'$, where $G=g+ g' + b-1$ and  $K = k + k' - (2p + b-1)$.   
\end{lemma}

\begin{proof}
Since there is an orientation-reversing diffeomorphism  $f : \partial W \rightarrow \partial W'$ which takes  the open book $\OB$ to the open book $\OB'$, we have $p'=p$ and $b'=b$.   Let $W=W_1 \cup W_2 \cup W_3$ and $W'=W_1' \cup W_2'\cup W_3'$ be $(g, k;p,b)$- and $(g', k';p,b)$-relative trisections, respectively. Then, by 
Definition~\ref{def: rela},  there are  diffeomorphisms $\varphi_i : W_i \to  Z_{k} = \natural^k S^1 \times B^3$ and $\varphi'_i : W_i'  \to Z_{k'} = \natural^{k'} S^1 \times B^3$ for $i =1,2,3$. 

Let $B\subset \del W$ and $B' \subset \del W'$ be the bindings of $\OB $ and $\OB'$, where $\pi : \del W \setminus B \to S^1$ and $\pi' : \del W' \setminus B' \to S^1$ are the projection maps of these open books, respectively.  Since the gluing diffeomorphism $f: \del W \to \del W'$ takes $\OB$ to $\OB'$ by our assumption,   we have $\pi' (f(\pi^{-1}(t)))=t$, for all $t \in S^1.$  Moreover,  we may assume that 
$$f_i = \varphi_i' \circ f \circ \varphi_i^{-1} :  Y^0_{g, k;p,b} \to Y^0_{g', k';p,b}$$   is a diffeomorphism for each $i=1,2,3$.  Informally, we identify each third of $\OB$ on $\del W$  with the appropriate third of $\OB'$ on $\del W'$ via the gluing map $f$.  This allows us to define $$X_i = W_i \underset{f}{\cup} W_i' = W_i \cup W_i' \big/ \sim$$ where $ x \sim y$ if  $x \in  \varphi^{-1}_i(Y^0_{g, k;p,b}) \subset \del W_i$ and $y =f(x)  \in \del W_i' $, where $\varphi_i' (y) \in  Y^0_{g', k';p,b}$. 

We claim that $X= W \cup_f W' = X_1 \cup X_2 \cup X_3$ is a $(G,K)$-trisection, where 
$$G=g+ g' + b-1 \;   \mbox{and} \;  K = k + k' - (2p + b-1).$$  In order to prove our claim,  we first need to describe, for each $i=1,2,3$,  a diffeomorphism $\Phi_i  : X_i  \to Z_K = \natural^K S^1 \times B^3$.  The diffeomorphism $\Phi_i$ is essentially obtained by gluing the diffeomorphisms $\varphi_i : W_i \to  Z_{k}$ and $\varphi_i' : W_i' \to  Z_{k'}$ using the diffeomorphism $f_i: \del W_i \to \del W_i'$, as we describe below.  

To  construct the desired diffeomorphism $\Phi_i  : X_i  \to Z_K = \natural^K S^1 \times B^3$, it suffices to describe how to glue $Z_k$ with $Z_{k'}$ to obtain $Z_K$ by identifying  $Y^0_{g, k;p,b} \subset \del Z_k$  with  $Y^0_{g', k';p,b} \subset \del Z_{k'}$ using the gluing map $f_i$.  

By definition,  $Z_k = U \natural V_n$, and similarly $Z_{k'}= U' \natural V_{n'}$, where $U = P \times D$ and $U' = P' \times D$.  Note that $P$ is diffeomorphic to $P'$ via $f_i$. To glue $Z_k$  to $Z_{k'}$ we identify $Y^0_{g, k;p,b} \subset \del U$ with  $Y^0_{g', k';p,b} \subset \del U'$ via the diffeomorphism $f_i : Y^0_{g, k;p,b} \to Y^0_{g', k';p,b}$.  Next, we observe that 
by gluing $U$ and $U'$ along the aforementioned parts of their boundaries using $f_i$,  we get $\natural^{l} S^1 \times B^3$, where $l =2p+b-1$.  

To see this,  we view $ U = P \times D$ as $P \times I_1 \times I_2$, and similarly $ U' = P' \times D$ as $P' \times I_1 \times I_2$, where $I_1=I_2=[0,1]$.    
We glue  $P \times I_1 $ with $P' \times I_1$ and then take its product with $I_2$. To glue $P \times I_1$ with $P' \times I_1$ we identify  $P \times \{1\}$ with $P' \times \{1\}$ using $f_i$. The result of  this identification is diffeomorphic to $P \times [0,2] \cong \natural^{l} S^1 \times B^2$. However, to complete the identification dictated by $f_i$, we have to take the quotient of  $P \times [0,2]$ by the relation $(x,t) \sim (x, 2-t)$ for all $x \in \del P$.  Note that we suppressed $f_i$ here since we have already identified $P$ with $P'$ via $f_i$.   The result is still diffeomorphic to the handlebody $\natural^{l} S^1 \times B^2$. Therefore, the gluing of $U$ and $U'$ is diffeomorphic to  $\natural^{l} S^1 \times B^3 $, since it is a thickening of $\natural^{l} S^1 \times B^2$ by taking its product with $I_2$.  

As a consequence,  the result of gluing $Z_k$ to $Z_{k'}$ is diffeomorphic to $$(\natural^{l} S^1 \times B^3) \natural V_n \natural V_{n'} \cong \natural^{l+n+n'} S^1 \times B^3 \cong \natural^K S^1 \times B^3 \cong Z_K, $$  since $l+n+n' = 2p+b-1 + k-(2p+b-1)+k'-(2p+b-1)= k+k'-(2p+b-1)=K$.

To finish the proof of the lemma,  we need to show that taking indices mod  $3$,  $$\Phi_i (X_i \cap X_{i+1}) = Y^+_{G,K} \; \mbox{and}\; \Phi_i (X_i \cap X_{i-1}) = Y^-_{G,K}, $$ where  $Y^+_{G,K} \cup Y^-_{G,K}$ is  the standard genus $G$ Heegaard splitting of $ \#^KS^1\times S^2 = Y_K  = \del Z_K$.

We observe that $$\Phi_i (X_i \cap X_{i+1}) = \varphi_i (W_i \cap W_{i+1})  \underset{f_i}{\cup}  \varphi_i' (W_i' \cap W_{i+1}') =  \big((Y^+_{g,k;p,b} \cup Y^+_{g', k';p,b}) \big/ \sim\big) \subset Y_K$$ where $x\sim y$ if $x \in (\partial P \times \partial^+D) \cup (P \times \{e^{i\pi/3}\}) \subset Y^+_{g,k;p,b}$  and $ y =f_i(x) \in Y^+_{g', k';p,b}$. 

Note that by definition, $Y^+_{g, k;p,b} = \partial^+U \natural \partial^+_sV_n$ and similarly  $Y^+_{g', k';p,b}= \partial^+U' \natural \partial^+_{s'}V_{n'}$. Since the boundary connected sums are taken along the interior of Heegaard surfaces, the identification $\sim$ does not interact with $\partial^+_sV_n$ and $\partial^+_{s'}V_{n'}$ and hence 
$$ (Y^+_{g,k;p,b} \cup Y^+_{g', k';p,b}) \big/\sim \;   \cong \big((\partial^+U \cup \partial^+U' ) \big/\sim\big) \natural (\partial^+_{s}V_{n} \natural \partial^+_{s'}V_{n'}).$$

But we see  that $(\partial^+U \cup \partial^+U') \big/\sim$  is diffeomorphic to $\natural^{l} S^1 \times B^2$, by exactly the same argument used above when we discussed the gluing of $U$ with $U'$.  Therefore, we have
	\begin{align*}
	(Y^+_{g,k;p,b} \cup Y^+_{g', k';p,b}) \big/\sim & \cong \big((\partial^+U \cup \partial^+U') \big/ \sim\big) \natural (\partial^+_{s}V_{n} \natural \partial^+_{s'}V_{n'})\\
			& \cong (\natural^lS^1\times B^2) \natural (\partial^+_{s}V_{n} \natural \partial^+_{s'}V_{n'})\\
			&\cong (\natural^lS^1\times B^2)\natural^{s+n+s'+n'}S^1\times B^2\\
			&\cong \natural^G S^1\times B^2
\end{align*}
since $l+s+n+s'+n'= g+g'+b-1=G$. 
Similarly, we have    $$\Phi_i (X_i \cap X_{i-1}) =  \varphi_i (W_i \cap W_{i-1})  \underset{f_i}{\cup}  \varphi_i' (W_i' \cap W_{i-1}')= \big(Y^-_{g,k;p,b} \cup Y^-_{g', k';p,b}) \big/ \sim\big) \subset Y_K$$ where $x\sim y$ if $x \in (\partial P \times \partial^-D) \cup (P \times \{e^{-i\pi/3}\}) \subset Y^-_{g,k;p,b}$  and $ y =f_i(x) \in Y^-_{g', k';p,b}$. Thus we obtain $(Y^-_{g,k;p,b} \cup Y^-_{g', k';p,b}) \big/\sim \;  \cong \natural^GS^1\times B^2$. Moreover,   $$\big((Y^+_{g,k;p,b} \cup Y^+_{g', k';p,b}) \big/ \sim\big) \cup_\del  \big((Y^-_{g,k;p,b} \cup Y^-_{g', k';p,b}) \big/ \sim\big) = Y_K.$$  

Therefore,  $Y^+_{G,K} = (Y^+_{g,k;p,b} \cup Y^+_{g', k';p,b}) \big/ \sim$ and $Y^-_{G,K} = (Y^+_{g,k;p,b} \cup Y^+_{g', k';p,b}) \big/ \sim$ gives  the standard genus $G$ Heegaard splitting of $ \#^KS^1\times S^2 = Y_K$, as desired.

To summarize, we showed that there is a diffeomorphism $\Phi_i  : X_i  \to Z_K = \natural^K S^1 \times B^3$, for each $i=1,2,3$, and  moreover, taking indices mod $3$, $\Phi_i (X_i \cap X_{i+1}) = Y^+_{G,K}$ and $\Phi_i (X_i \cap X_{i-1}) = Y^-_{G,K}$. Therefore, we conclude that $X= W \cup_f W' = X_1 \cup X_2 \cup X_3$ is a $(G,K)$-trisection.  \end{proof}

Here is an immediate corollary of Lemma~\ref{lem:gluem}.

\begin{corollary} \label{cor: double} Suppose that $W=W_1 \cup W_2 \cup W_3$ is a $(g, k;p,b)$-relative trisection of a $4$-manifold $W$ with nonempty connected boundary. Let $DW$ denote the double of $W,$ obtained by gluing $W$ and $\overline{W}$ (meaning $W$ with the opposite orientation) by the identity map of  the boundary $\del W$. Then $DW$ admits a $(2g+b-1, 2k-2p-b+1)$-trisection. 
\end{corollary}
\begin{proof} 
 If $W=W_1 \cup W_2 \cup W_3$ is a $(g, k;p,b)$-relative trisection of $W$ with the induced open book  $\OB$ on $\del W$, then  $\overline W=\overline W_1 \cup \overline W_2 \cup \overline W_3$ is a $(g, k;p,b)$-relative trisection of $\overline W$ with the induced open book  $\overline \OB$ on $\del \overline W$, where $\overline \OB$ is obtained from $\OB$ by  reversing the orientation of the pages. Since  the identity map from $\del W$ to $\del \overline W$  is an orientation-reversing diffeomorphism which takes $\OB$ to $\overline \OB$, we  obtain the desired result about $DW$ by  Lemma~\ref{lem:gluem}. \end{proof} 

The point of Corollary~\ref{cor: double} is that one does not need to know the monodromy of the open book on $\del W,$  to describe a trisection on $DW$.

\begin{example} \label{ex: eng} Let $E_{n,h}$ denote the $B^2$-bundle over $\Si_h$ with Euler number $n \in \mathbb{Z}$.   In \cite{cgp}, there is a description of a $(|n| + h, |n| + 2h-1; h, |n|)$-relative trisection of $E_{n,h}$  for $n \neq 0$, and a  $(h+2, 2h+1; h, 2)$-relative trisection of $E_{0,h}=\Si_h \times B^2 $.  Since the double of $E_{n,h}$ is $\Si_h \times S^2$ or $ \Si_h \tilde{\times} S^2$ depending on $n$ modulo $2$, we get a $(2h+3|n|-1, 2h+|n|-1)$-trisection of $\Si_h \times S^2$ (resp. $ \Si_h \tilde{\times} S^2$) for any even (resp. odd) nonzero integer $n$, by Corollary~\ref{cor: double}.  

In particular,  by doubling the  $(h+2, 2h+1; h, 2)$-relative trisection of $\Si_h \times B^2 $, we obtain a $(2h+5, 2h+1)$-trisection of $\Si_h \times S^2$ which is smaller compared to the $(8h+5, 4h+1)$-trisection presented in \cite{gk}, provided that $h \geq 1$.  Similarly, by setting $n=\pm1$, we obtain a  $(2h+2, 2h)$-trisection for $ \Si_h \tilde{\times} S^2$, which is not covered by the examples in \cite{gk}, except for $h=0$.  Note that there is also a $(2,1;0,2)$-relative trisection of $E_{n,0}$ given in \cite{cgp} for each $n \in \mathbb{Z}$.   Since the double of $E_{n,0}$ is $S^2 \times S^2$ or $S^2 \tilde{\times} S^2$ depending on $n$ modulo $2$, we get {\em infinitely many} $(5,1)$-trisections of $S^2 \times S^2$ and $S^2 \tilde{\times} S^2$. \end{example}

\begin{corollary} \label{cor: euler} If $W=W_1 \cup W_2 \cup W_3$ is a $(g, k;p,b)$-relative trisection of a $4$-manifold $W$ with nonempty connected boundary, then the Euler characteristic $\chi(W)$ is equal to $g-3k+3p+2b-1.$
\end{corollary}

\begin{proof}    
Using Corollary~\ref{cor: double}, we compute $$\chi(W)=\frac{1}{2}\chi(DW)= \frac{1}{2} (2+ 2g+b-1 -3(2k-2p-b+1)))= g-3k+3p+2b-1. $$One can of course derive the same formula directly from the definition of a  relative trisection.   \end{proof}

Since every relatively trisected $4$-manifold with connected boundary is determined by some relative trisection diagram, it would be desirable to have a version of Lemma ~\ref{lem:gluem}, where one ``glues" the relative trisection diagrams corresponding to $W=W_1 \cup W_2 \cup W_3$ and $W'=W_1' \cup W_2' \cup W_3'$  to get a diagram corresponding to the trisection  $X= W \cup_f W' =X_1 \cup X_2 \cup X_3$.  This is the content of Proposition~\ref{prop:gluediagram}, but first we develop some language to be used in its statement.

Let  $(\Si, \a, \b, \g)$ and $(\Si', \a', \b', \g')$ be  $(g, k;p,b)$- and $(g', k';p,b)$-relative trisection diagrams corresponding to  the relative trisections $W= W_1 \cup W_2 \cup W_3$ and $W'=W_1' \cup W_2' \cup W_3'$, with induced open books $\OB$ and $\OB'$ on $\del W$ and $\del W'$, respectively.  Suppose that there is an orientation-reversing diffeomorphism $f:\partial W \rightarrow \partial W'$ which  takes $\OB$ to $\OB'$.  We observe that  since the orientation-reversing diffeomorphism $f : \del W \to \del W'$ takes $\OB$ to $\OB'$, it descends to an orientation-reversing  diffeomorphism, denoted again by $f$ for simplicity,  from  the page $\Si_\a$ of $\OB$ onto the page $\Si'_{\a'}$ of $\OB'$. Therefore,  $f$ restricted to $\del \Si_\a$ is an orientation-reversing diffeomorphism from the binding $\del \Si_\a$ of $\OB$ to the binding $\del \Si'_{\a'}$ of $\OB'$.

Let $(\A_\alpha, \A_\beta, \A_\gamma)$  denote a cut system of arcs  as in Definition~\ref{def: cutarc} associated to the diagram $ (\Sigma, \alpha, \beta, \gamma)$.  We denote the arcs in $\A_\a$ as $\{ a_1, \ldots, a_l\}$, the arcs in $\A_\b$ as $\{ b_1, \ldots, b_l\}$, and the arcs in $\A_\g$ as $\{ c_1, \ldots, c_l\}$, where $l= 2p+b-1$.  

We choose a cut system $(\A_{\a'}, \A_{\b'}, \A_{\g'})$ of arcs  associated to the diagram $ (\Sigma', \alpha', \beta', \gamma')$
as follows.  Since the collection $\A_\a$ of arcs cuts $\Si_{\a}$ into a disk, the collection $\A_{\a'} = \{ a'_1, \ldots, a'_l\} $ of arcs, where $a'_i = f(a_i)$ for $i=1, \ldots, l$,  cuts $\Si'_{\a'}$ into a disk as well.  Then we obtain $\A_{\b'}= \{ b'_1, \ldots ,  b'_l\}$ and $\A_{\g'} =  \{ c'_1, \ldots  ,  c'_l\}$ from $\A_{\a'}$ as in Definition~\ref{def: cutarc}.  In particular, we see that $\del {a'_i} = f(\del {a_i})$, $\; \del {b'_i} = f(\del {b_i})$  and $\del {c'_i} = f(\del {c_i})$ for  each $i = 1, \ldots, l$. 

\begin{definition} \label{def:star} Let  $\Sigma^* $ denote the closed, oriented  surface obtained by gluing  $\Sigma$ and $\Si'$ along their boundaries using the orientation-reversing diffeomorphism $f : \del \Si = \del \Si_\a \to \del \Si'_{\a'} =  \del \Si'$ defined above.  It follows that   $\overline{\a}_i = a_i \cup_\partial a'_i$, $\;\overline{\b}_i =b_i \cup_\partial b'_i$ and $\overline{\g}_i = c_i \cup_\partial c'_i$ are simple closed curves in $\Si^*$, for $i = 1, \ldots, l$. Then the collection of $G= g+g'+b-1$ disjoint, simple closed curves $\a^* = \{ \a_1, \ldots, \a_G\}  \subset \Si^*$ is defined as follows
$$\alpha^*_i:= \left\{\begin{array}{cl}
					\a_i  & 1\leq i\leq g-p \cr
					\overline{\a}_{i-g+ p}  & g-p+1\leq i\leq G-g'+p\cr
					\alpha_i' & G-g'+p+1 \leq i\leq G\\
	\end{array}\right.$$  
We write  $\a^* = \a \cup \overline{\a} \cup \a'$. The collection of curves $\b^* = \b \cup \overline{\b} \cup \b'$ and $\g^* = \g \cup \overline{\g} \cup \g'$ are defined similarly.  We say that $(\Sigma^*, \alpha^*, \beta^*, \gamma^*)$  is obtained by gluing $(\Si, \a, \b, \g)$ and $(\Si', \a', \b', \g')$ by the map $f$. \end{definition}

\begin{proposition}\label{prop:gluediagram}  Let   $W$ and $W'$ be $4$-manifolds such that $\partial W$ and $\partial W'$ are both nonempty and connected.  Suppose that $(\Sigma, \alpha, \beta, \gamma)$ and $(\Sigma', \alpha', \beta', \gamma')$ are   $(g, k;p,b)$- and $(g', k';p,b)$-relative trisection diagrams corresponding to  the relative trisections  $W= W_1 \cup W_2 \cup W_3$ and $W'=W_1' \cup W_2' \cup W_3'$, with induced open books $\OB$ and $\OB'$ on $\partial W$ and  $\partial W'$,  respectively.
Suppose further that there is an orientation-reversing diffeomorphism $f:\partial W \rightarrow \partial W'$ which  takes $\OB$ to $\OB'$.  Then $(\Sigma^*, \alpha^*, \beta^*, \gamma^*)$, which is obtained by gluing $(\Sigma, \alpha, \beta, \gamma)$ and $(\Sigma', \alpha', \beta', \gamma')$ by the map $f$ as in Definition~\ref{def:star},  is a $(G,K)$-trisection diagram corresponding to the trisection $X= W \cup_f W' =X_1 \cup X_2 \cup X_3$ described in Lemma~\ref{lem:gluem}, where $G=g+g'+b-1$ and $K= k+k' -(2p+b-1)$.   \end{proposition}

\begin{proof} We claim that $(\Sigma^*, \alpha^*, \beta^*, \gamma^*)$ is a $(G,K)$-trisection diagram representing the trisection $X= W \cup_f W' = X_1 \cup  X_2 \cup X_3$ described in Lemma~\ref{lem:gluem}. By construction, $\Si^*$ is a closed, oriented  surface of  genus $$ (g-p) + (2p+b-1) + (g'-p) = g+g'+b-1 =G$$ and each of $\a^*$, $\b^*$ and $\g^*$ is a non-separating collection of $G$ disjoint simple closed curves on $\Si^*$.   To finish the proof, we need to show that $\a^*$ bounds disks in $X_1 \cap X_2$,  $\b^*$ bounds disks in $X_2 \cap X_3$,  $\g^*$ bounds disks in $X_1 \cap X_3$.  

We know that $\a$ curves bound disks in $W_1 \cap W_2$,  $\b$ curves bound disks in $W_2 \cap W_3$,  and $\g$ curves bound disks in $W_1 \cap W_3$. Similarly,  $\a'$ curves bound disks in $W'_1 \cap W'_2$,  $\b'$ curves bound disks in $W'_2 \cap W'_3$,  and $\g'$ curves bound disks in $W'_1 \cap W'_3$.  Therefore $\a \cup \a'$ curves bound disks in $X_1 \cap X_2$,  $\b \cup \b'$ curves bound disks in $X_2 \cap X_3$,  and $\g \cup \g'$ curves bound disks in $X_1 \cap X_3$.  

Hence, all we need to show is that $\overline{\a}$ curves  bound disks in $X_1 \cap X_2$,  $\overline{\b}$ curves bound disks in $X_2 \cap X_3$,  and 
$\overline{\g}$ curves bound disks in $X_1 \cap X_3$. But this follows by the fact that  $\overline{\alpha} $ curves bound disks in the handlebody $\partial^+U \cup \partial^+U'/\sim$,  whereas $\overline{\beta} $ curves bound disks in the handlebody $\partial^-U \cup \partial^-U'/\sim$.  Similar statement holds for the pairs $(\overline{\b}, \overline{\g})$ and $(\overline{\a}, \overline{\g})$. \end{proof}


\begin{corollary} \label{cor: diag} Suppose that $W=W_1 \cup W_2 \cup W_3$ is a $(g, k;p,b)$-relative trisection of a $4$-manifold $W$  with nonempty connected boundary and let $(\Si, \a, \b, \g)$ be a corresponding relative trisection diagram.  Then the $(2g+b-1, 2k -2p-b+1)$-trisection of the double $DW$ of $W$ described in Corollary~\ref{cor: double} has a corresponding diagram $(\Si^*, \a^*, \beta^*, \g^*)$ which is obtained by gluing $(\Si, \a, \b, \g)$ and   $(\overline{\Si}, \a, \b, \g)$ by the identity map  from $\del W$ to $\del \overline W$.  \end{corollary}

\begin{proof}  If $W=W_1 \cup W_2 \cup W_3$ is a $(g, k;p,b)$-relative trisection of $W$ with the induced open book  $\OB$ on $\del W$, whose corresponding relative trisection diagram is $(\Si, \a, \b, \g)$, then  $\overline W=\overline W_1 \cup \overline W_2 \cup \overline W_3$ is a $(g, k;p,b)$-relative trisection of $\overline W$ with the induced open book  $\overline \OB$ on $\del \overline W$, whose corresponding relative trisection diagram can be canonically given by  $(\overline \Si, \a, \b, \g)$.  The proof is complete by observing that  the identity map from $\del W$ to $\del \overline W$  is an orientation-reversing diffeomorphism  taking $\OB$ to $\overline \OB$.  \end{proof}

\begin{remark}  In Sections~\ref{subsec: trivial}  and \ref{subsec: twisted}, using Corollary~\ref{cor: diag}, we  draw explicit diagrams  for the $(7, 3)$-trisection of $T^2 \times S^2$  and the $(4, 2)$-trisection of  $T^2 \tilde{\times} S^2$ given in Example~\ref{ex: eng}, respectively.       \end{remark}

\section{Trisecting Lefschetz fibrations} \label{sec: main}
In this section, we  prove Theorem~\ref{thm: main}.  We refer to \cite{gs, os} for the definitions and properties of Lefschetz fibrations and  open books.   We need some preliminary results. The following lemma is well-known (cf. \cite{ao}). 

\begin{lemma} \label{lem: palf} Suppose that a closed $4$-manifold $X$ admits
a genus $p$ Lefschetz fibration $\pi_X :  X \to S^2$ with a section of square $-1$ so that the 
vanishing cycles of $\pi_X$ are given by the ordered set of simple closed curves  
$\{ \l_1, \l_2, \ldots, \l_n\}$. Let $V$ denote a regular neighborhood of the section union a nonsingular fiber and let $W$ 
denote the $4$-manifold with boundary
obtained from $X$ by removing the interior of $V$. Then $\pi_X: X \to S^2$ descends to a Lefschetz fibration $\pi_W: W \to B^2$ whose regular fiber is a genus $p$ surface $\Si_{p,1}$  with connected boundary. Moreover, the monodromy of $\pi_W$ (and hence the open book naturally induced on $\del W$) is given by 
$$D(\d) = D(\l_n) D(\l_{n-1}) \cdots D(\l_1)$$ where $\d$ is a boundary parallel curve in $\Si_{p,1}$.  
\end{lemma}

\begin{lemma} \label{lem: plum}  Let $V$ denote a regular neighborhood of the section of square $-1$ union a nonsingular fiber as in Lemma~\ref{lem: palf}. Then there is an achiral Lefschetz fibration $\pi_V:  V \to B^2$ whose regular fiber is a surface $\Si_{p,1}$  of genus $p$ with connected boundary such that $\pi_V$ has only one singular fiber carrying two singularities. Moreover,  the monodromy of the open book on $\del V$ induced by $\pi_V$ is a single left-handed Dehn twist along a boundary parallel curve in $\Si_{p,1}$. 
\end{lemma}

\begin{proof} Let $\Si_p$ denote the regular fiber of the Lefschetz fibration $\pi_X : X \to S^2$ constructed in Lemma~\ref{lem: palf}. Then the 
$4$-manifold $V$ can be described as the  plumbing of the disk bundle over $S^2$ with Euler number $-1$ with the trivial disk bundle $B^2 \times \Si_p$. Applying the algorithm in \cite[Section 5.3]{cgp}, we obtain an achiral  Lefschetz fibration $\pi_V : V \to B^2$ whose regular fiber is a surface $\Si_{p,1}$ of genus $p$ with connected boundary. Moreover, $\pi_V$ has only one singular fiber carrying two vanishing cycles: a homotopically trivial curve $\e$ (so $\pi_V$ is not relatively minimal) with framing $-1$, and a boundary parallel curve $\d$ in $\Si_{p,1}$ with framing $+1$. It follows that the monodromy of  the open book on $\del V$ is a single left-handed Dehn twist 
$D^{-1}(\d)$.  \end{proof} 

\begin{remark} \label{blowd} Let $\widetilde{V}$ denote the $4$-manifold obtained from $V$ by blowing down the $(-1)$-sphere. Then there is an achiral Lefschetz fibration $\pi_{\widetilde{V}} : \widetilde{V} \to B^2$ whose regular fiber is $\Si_{p,1}$, and which contains only one singular fiber whose vanishing cycle is $\delta$. This can be most easily seen by drawing a Kirby diagram of $V$ and simply blowing down the sphere with framing $-1$.   
\end{remark}

\begin{lemma} \label{lem: leftri} \cite[Corollary 17]{cgp} Let $\pi_W : W  \to B^2$  be an achiral Lefschetz fibration with regular fiber a surface $\Si_{p,b}$  of genus $p$ with $b$ boundary components and with $n$
vanishing cycles. Then there is a $(p+n, 2p+b-1; p, b)$-relative trisection of $W$ realizing the natural open book on $\del W$ induced from the Lefschetz fibration $\pi_W$. Moreover, the corresponding trisection diagram can be described explicitly, based on the vanishing cycles of $\pi_W$.   
\end{lemma}

Here we briefly sketch a proof of Lemma~\ref{lem: leftri}.   We start with describing a relative trisection of the neighborhood  $\Si_{p,b} \times B^2$ of a nonsingular fiber $\Si_{p,b}$ in $\pi_W$.  Let $\pi: \Si_{p,b} \times B^2 \to  B^2$ denote the projection onto the second factor. Trisecting  the base $B^2$ into three wedges and taking the union of the inverse images of these pieces under $\pi$ gives a  $(p, 2p+b-1; p,b)$-relative trisection of $\Si_{p,b} \times B^2$, such that the open book on 
 $ \del ( \Si_{p,b} \times B^2 )$ is the trivial  one, whose page is $\Si_{p,b}$ and monodromy is the identity map. Note that the open book on  $ \del ( \Si_{p,b} \times B^2 )$   induced by the trivial fibration  $\Si_{p,b} \times B^2 \to B^2$ is the same as the one induced by the relative trisection.  Moreover, the corresponding relative trisection diagram is empty, i.e.,  it is a genus $p$ surface with $b$ boundary components with no $\a$, $\b$ or $\g$ curves on it.

It is well-known that the total space $W$ of the achiral Lefschetz fibration $\pi_W : W  \to B^2$ is obtained by attaching a $2$-handle to the product $\Si_{p,b} \times B^2$ for each vanishing cycle.  The $2$-handle is attached along the vanishing cycle with framing $\pm 1$ with respect to the surface framing. Therefore,  to find a  description of  a relative trisection of $W,$  it suffices  to extend  the relative trisection on  $\Si_{p,b} \times B^2$ over any $2$-handle attachment as described above.   It turns out that such an extension is possible  in a more general setting (cf. \cite[Lemma 15]{cgp}), and we describe its diagrammatic version below. 

Suppose that  $(\Si, \a, \b, \g)$ is a relative trisection diagram of a $4$-manifold $N$ which admits an achiral Lefschetz fibration $\pi_N : N \to B^2$ such that the open books induced by the relative trisection and the Lefschetz fibration agree on $\del N$. Let $\l$ be a simple closed curve on $\Si$ which is disjoint from $\a$ and transverse to  $\b$ and $\g$.  Hence, we can view $\l$ as a curve on $\Si_\a$, which we identified as the page of the aforementioned open book on $\del N$.  If we attach a $2$-handle to $N$ along $\l$ with framing $\pm 1$ with respect to the page $\Si_\a$, then a relative trisection diagram $(\widetilde{\Si}, \widetilde{\a}, \widetilde{\b}, \widetilde{\g}^\pm)$ of the resulting $4$-manifold is described as follows.

\begin{figure}[ht]
\includegraphics[scale=.35]{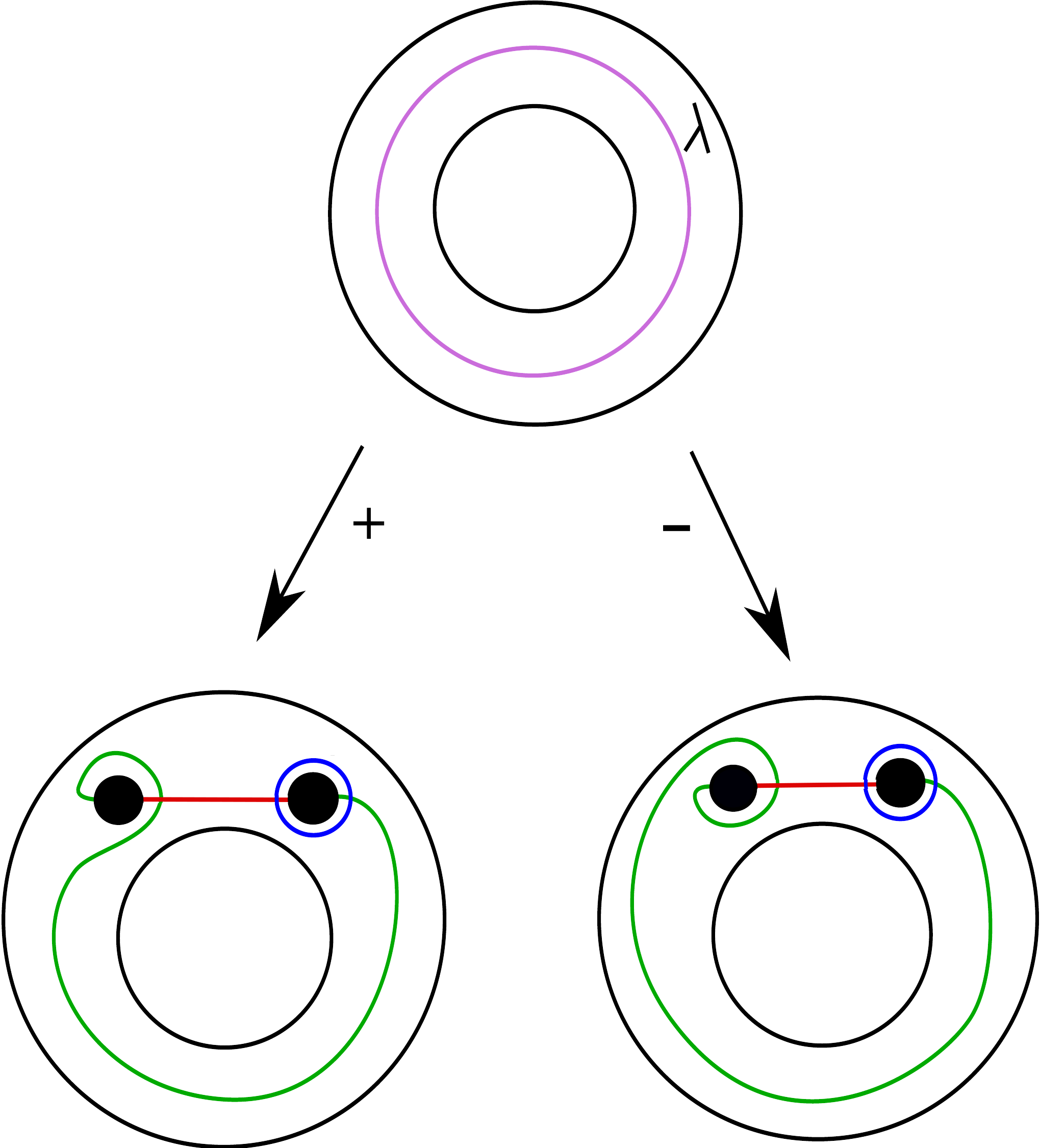}
	\caption{On top: annular neighborhood of the vanishing cycle $\l \subset \Si$. At the bottom:  removing each  pair of black disks and gluing in  a cylinder we get two copies of a two-holed torus --- one on the left and one on the right. The red and green arcs are extended over the cylinders as usual, so that both are simple closed curves on the surface $\widetilde{\Si}$.  ~\label{fig:lef}}
\end{figure}

The surface  $\widetilde{\Si}$ is obtained from $\Si$ by  removing  an annular neighborhood of $\l$ from $\Si$ and inserting  in  a two-holed torus. The two-holed torus carries three curves (colored red, blue and green) with two options for the green curve corresponding to the attaching framing of the $2$-handle,  as shown in Figure~\ref{fig:lef}.  We define $ \widetilde{\a}$ to be  $\a$ union  the new red curve, $ \widetilde{\b}$ to be  $\b$ union  the new blue curve, and $ \widetilde{\g}^\pm$ to be  $\g$ union the new green curve on the bottom left (resp. right) in Figure~\ref{fig:lef}. 

\begin{remark} \label{rem: order} In order to draw the relative trisection diagram of  an achiral  Lefschetz fibration over $B^2$ described by the {\em ordered} set of vanishing cycles $\l_1, \ldots, \l_n \subset \Si_{p,b}$,  we proceed as follows. First we draw $\l_1$ on the surface $\Si_{p,b}$ and replace an annular  neighborhood of it by one of the two-holed tori shown at the bottom of Figure~\ref{fig:lef} depending on the sign of the surgery coefficient, to obtain a surface $\Si_{p+1,b}$ decorated with one red, one blue and one green curve. Then we isotope $\l_2$  so that it does not intersect the red or the blue curve but only intersects the green curve transversely, say $m \geq 0$ times,  in the diagram. Next,  we remove an annular neighborhood of $\l_2$, and therefore  we also remove $m$  disjoint arcs of the green curve. Now we plug in  the appropriate  two-holed torus shown in Figure~\ref{fig:lef}, so that each arc cut from the first green curve is replaced by a green arc in the glued in two-holed torus with the same end points,  disjoint from both the new blue and the green curves and transversely intersecting the new red curve once.  As a result, we obtain a surface $\Si_{p+2,b}$ decorated with two red curves, two blue curves and two green curves,  so that two curves with the same color are disjoint. Then we isotope $\l_3$  such  that it does not intersect the red or the blue curves but only intersects the green curves transversely  and  apply the same procedure as we implemented  for $\l_2$. It is now clear how to iterate this procedure for the rest of the vanishing cycles. It is important to note that for a different ordering of the same set of vanishing cycles we get a different relative trisection  diagram, in general. This is of course consistent with the fact that the total space of the Lefschetz fibration (and therefore the open book on the boundary) is determined by the {\em ordered}  set of vanishing cycles.  \end{remark}

\begin{remark}  The Euler characteristic $\chi(W)$ can be calculated simply as $\chi(B^2)\chi (\Si_{p,b}) +n=2-2p-b +n$, using the achiral Lefschetz fibration $\pi_W : W  \to B^2$ described in Lemma~\ref{lem: leftri}.  On the other hand, using the $(p+n, 2p+b-1; p, b)$-relative trisection of $W$, we have $\chi(W) = p+n - 3(2p+b-1) + 3p + 2b -1= 2-2p-b+n$, as expected, by the formula in Corollary~\ref{cor: euler}. 

\end{remark}

\begin{theorem} \label{thm: main} Suppose that $X$ is a  smooth, closed, oriented, connected $4$-manifold  which admits a genus $p$ Lefschetz fibration over $S^2$ with $n$ singular fibers, equipped with a section of square $-1$. Then, an explicit  $(2p+n+2, 2p)$-trisection of $X$ can be described  by a  corresponding trisection diagram,  which is determined by the vanishing cycles of the Lefschetz fibration.      Moreover, if $\widetilde{X}$ denotes the $4$-manifold obtained from $X$ by blowing down the section of square $-1$, then we also obtain a $(2p+n+1, 2p)$-trisection of $\widetilde{X}$ along with a corresponding diagram. \end{theorem}

 \begin{proof} Suppose that $\pi_X: X \to S^2$  is a genus $p$ Lefschetz fibration  with $n$ singular fibers, equipped with a section of square $-1$. Let $V$ denote a regular neighborhood of the section union a nonsingular fiber as in Lemma~\ref{lem: palf}.  Then the $4$-manifold $W$ obtained by removing the interior of $V$ from $X$ admits a  Lefschetz fibration $\pi_W: W \to B^2$ whose regular fiber is a genus $p$ surface $\Si_{p,1}$ with connected boundary. Note that the monodromy of the open book on the boundary $\del W $ (oriented as the boundary of $W$) is given by  $D(\d) \in \G_{p,1}$. Applying  Lemma~\ref{lem: leftri} we get a $(p+n, 2p; p, 1)$-relative trisection on $W$ realizing the open book on $\del W$.  

On the other hand, there is an achiral Lefschetz fibration $\pi _V : V \to B^2$ as described in Lemma~\ref{lem: plum}. It follows, by Lemma~\ref{lem: leftri}, that $V$ admits a $(p+2, 2p; p, 1)$-relative trisection  realizing the open book on $\del V$ (oriented as the boundary of $V$) whose monodromy is given by $D^{-1}(\d) \in \G_{p,1}$.

To get a $(2p+n+2, 2p)$-trisection on $X$ we just glue the $(p+n, 2p; p, 1)$-relative trisection on $W$ with the  $(p+2, 2p; p, 1)$-relative trisection on $V$ along the identical open book on their common boundary $\del W= -\del V$, using Lemma~\ref{lem:gluem}.  In addition, using Proposition~\ref{prop:gluediagram}, the corresponding relative trisection diagrams can be glued together diagrammatically to obtain a $(2p+n+2, 2p)$-trisection diagram of $X$.

To prove the last statement in Theorem~\ref{thm: main}, we first observe by Remark~\ref{blowd}  that 
$$X \cong W \cup_{\del} V \cong W  \cup_{\del} (\widetilde{V} \# \cpb) \cong (W \cup_{\del}  \widetilde{V})   \# \cpb \cong \widetilde{X} \# \cpb . $$ In particular, we have $\widetilde X  =  W \cup_{\del}  \widetilde{V}$, and hence  using Lemma~\ref{lem:gluem}  we obtain a  $(2p+n+1, 2p)$-trisection of $\widetilde X$ by gluing the  $(p+n, 2p; p, 1)$-relative trisection on $W$ and the $(p+1, 2p; p, 1)$-relative trisection on $\widetilde V$  corresponding to  
the Lefschetz fibration   $\pi_{\widetilde{V}} :  \widetilde{V} \to B^2$ with only one vanishing cycle,   along the identical open book on their boundaries. The corresponding trisection diagram for $\widetilde{X}$ can be obtained using Proposition~\ref{prop:gluediagram}. \end{proof}

\section{Trisection diagrams} \label{sec: ex}

In our figures, if there is no indicated boundary, then by adding the point at infinity we obtain a {\em closed} oriented surface.  Two small black disks labeled with the same white number or letter represents surgery on the sphere $S^0$ consisting of the centers of these disks. This means that we remove these two disks from the underlying oriented surface and glue in a cylinder $I \times S^1$   to get an oriented surface of one higher genus.  Any arc connecting two  black disks with the same label represents a simple closed curve obtained  by joining the two ends of this arc by $I \times p \subset I \times S^1$ for some $p \in S^1$.  Therefore, we refer to these kind of arcs as curves in our figures and reserve the term arc for the properly embedded arcs appearing in the cut systems.

In the trisection diagrams, the $\a, \b$ and $\g$ curves are drawn (and also referred to) as red, blue and green curves, respectively.   We follow the same coloring convention for the cut system $(\A_\a, \A_\b, \A_\g)$  of arcs for the relative trisection diagrams.  Namely, the arcs in $\A_\a$, $\A_\b$ and $\A_\g$ are  drawn (and also referred to) as red, blue and green arcs, respectively.

{\bf Notation:}  We use $\G_{p,b}$ to denote the mapping class group of the genus $p$ surface $\Si_{p,b}$ with $b$ boundary components, and $D(\l)$ to denote the right-handed Dehn twist along a simple closed curve $\l \subset \Si_{p,b}$. We use functional notation for the products of Dehn twists in $\G_{p,b}$. 

\subsection{The elliptic surface $E(1)$}

Our goal in this subsection is to  illustrate the method of proof of Theorem~\ref{thm: main}  by constructing an explicit $(16,2)$-trisection diagram of the elliptic surface $E(1)$, based on the standard elliptic fibration $E(1)  \to S^2$ with $12$ singular fibers. This is not the minimal genus trisection of $E(1)$, however, since $E(1)$ is diffeomorphic to $\cp \# 9 \cpb$, which  admits a genus $10$ trisection obtained from the connected sum of genus one trisections of $\cp$ and $\cpb$ given in \cite{gk}.  

It is well-known that the relation $D(\d) = (D(b)D(a))^6 \in  \G_{1,1}$,   where $a$ and $b$ denote  the standard generators of the first homology of $\Si_{1,1}$ and $\d$ is a curve  parallel to $\del \Si_{1,1}$, describes the elliptic Lefschetz fibration $E(1) \to S^2$ with $12$ singular fibers, equipped with a section of square $-1$.  

According to our notation in Section~\ref{sec: main}, $E(1) = W \cup_{\del} V$, where $V$ denotes a regular neighborhood of the section union a nonsingular fiber.  By Lemma~\ref{lem: plum}, there is  an achiral Lefschetz fibration  $\pi_V : V \to B^2$, with two vanishing cycles on $\Si_{1,1}$, a homotopically  trivial curve  $\e$ with framing $-1$ and a boundary parallel curve $\d$ with framing $+1$.

Note that the Dehn twist $D(\e)$ is isotopic to the identity since $\e$ is homotopically trivial, and hence it does not contribute to the monodromy.  However, we still have to take the vanishing cycle $\e$ into account while drawing the corresponding $(3,2;1,1)$-relative trisection diagram of $V$ shown in Figure~\ref{fig:LFPb}, which we obtained by applying  Remark~\ref{rem: order}. In Figure~\ref{fig:LFPb},  the surgeries labeled by $1$ and $2$ correspond to the vanishing cycles $\e$ and $\d$, respectively.  

Next, we  decorate the relative trisection diagram for $V$   with  a cut system  $(\A^V_\a, \A^V_\b, \A^V_\g)$ of arcs as follows.  First note that surgeries along the two red curves cancel out the surgeries labeled by $1$ and $2$  in Figure~\ref{fig:LFPb} and hence $\Sigma_\a$ is the genus one surface with one boundary component,  represented by the surgery labelled by $h$. 

By definition, the set $\A^V_\a$ consists of two red arcs that  cut $\Sigma_\a$ into a disk. An obvious choice of  $\A^V_\a$  is depicted in Figure~\ref{fig:VArcs}.  Then we obtain $\A^V_\b$ consisting of the two blue arcs in Figure~\ref{fig:VArcs},  simply by taking parallel copies of the red arcs.  We do not need any handleslides since the resulting blue arcs are clearly disjoint from the blue curves. Finally, the two green arcs belonging to $\A^V_\g$ in Figure~\ref{fig:VArcs} are  obtained by applying some handleslides to the parallel copies of the blue arcs over the blue curve associated to the surgery  labeled by $2$.

\begin{figure}
\centering
\begin{minipage}{.5\textwidth}
  \centering
  \includegraphics[width=.9\linewidth]{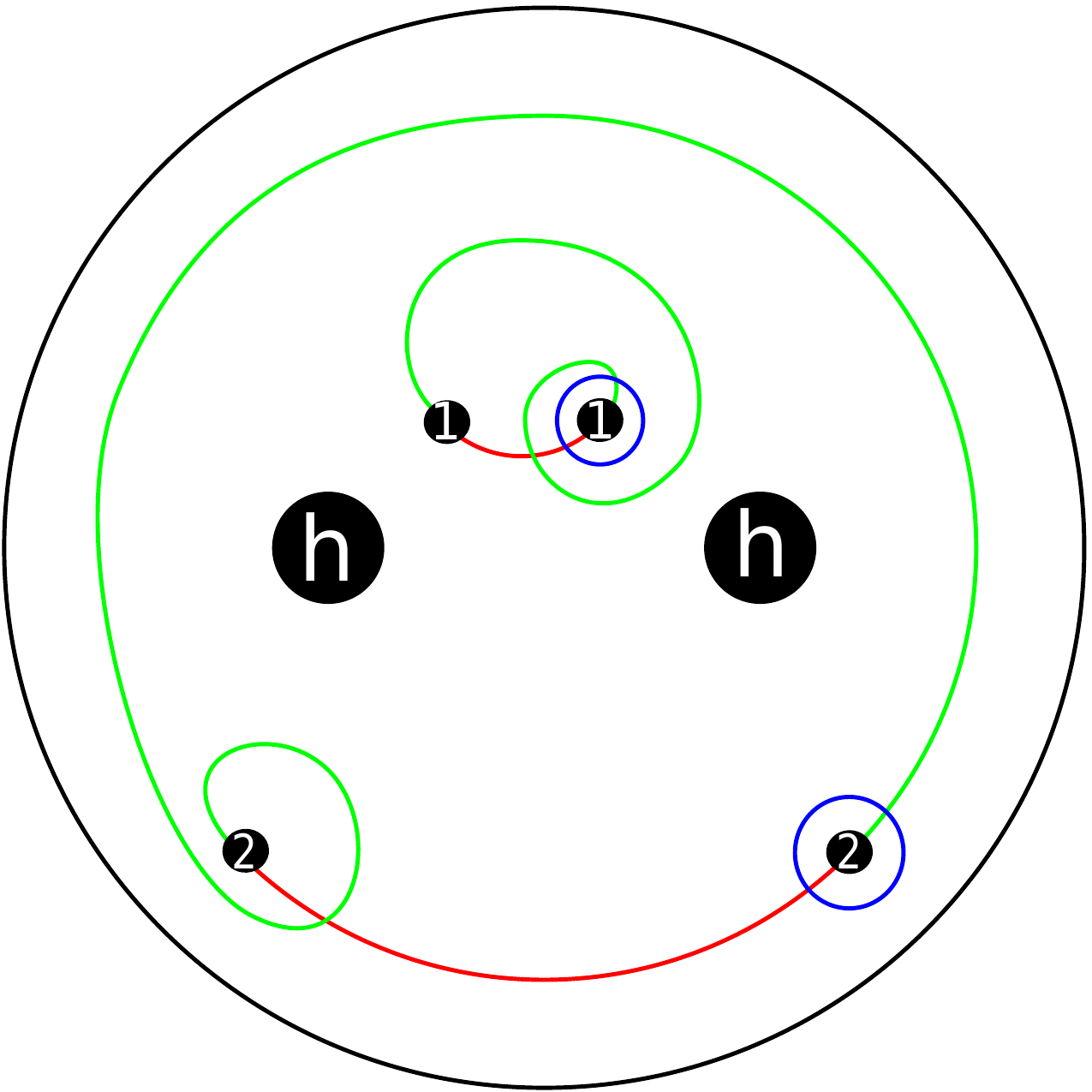}
  \captionof{figure}{The $(3,2;1,1)$-relative trisection diagram for $V.$}
  \label{fig:LFPb}
\end{minipage}%
\begin{minipage}{.5\textwidth}
  \centering
  \includegraphics[width=.9\linewidth]{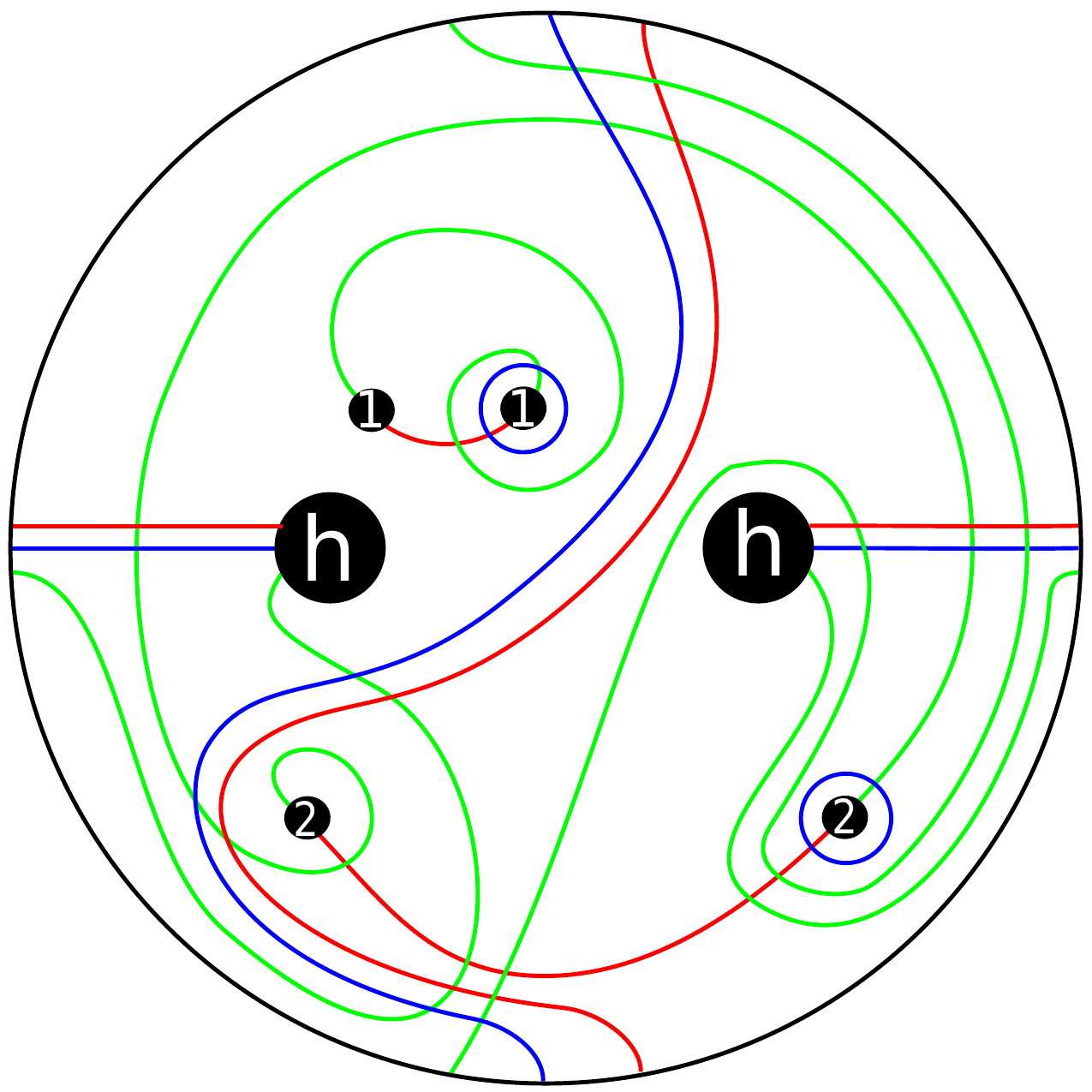}
  \captionof{figure}{The cut system  $(\A^W_\a, \A^W_\b, \A^W_\g)$ of arcs for the trisection diagram of $V.$}
  \label{fig:VArcs}
\end{minipage}
\end{figure}

\begin{figure}

	\labellist
		\pinlabel $b_1$ at 200 254
		\pinlabel $b_2$ at 208 209
		\pinlabel $b_3$ at 215 160
		\pinlabel $b_4$ at 215 110
		\pinlabel $b_5$ at 208 70
		\pinlabel $b_6$ at 200 24
		
		\pinlabel $a_6$ at 202 145 
		\pinlabel $a_5$ at 178 145
		\pinlabel $a_4$ at 155 145
		\pinlabel $a_3$ at 132 145
		\pinlabel $a_2$ at 110 145
		\pinlabel $a_1$ at 86 145
	\endlabellist
		\includegraphics[scale=.7]{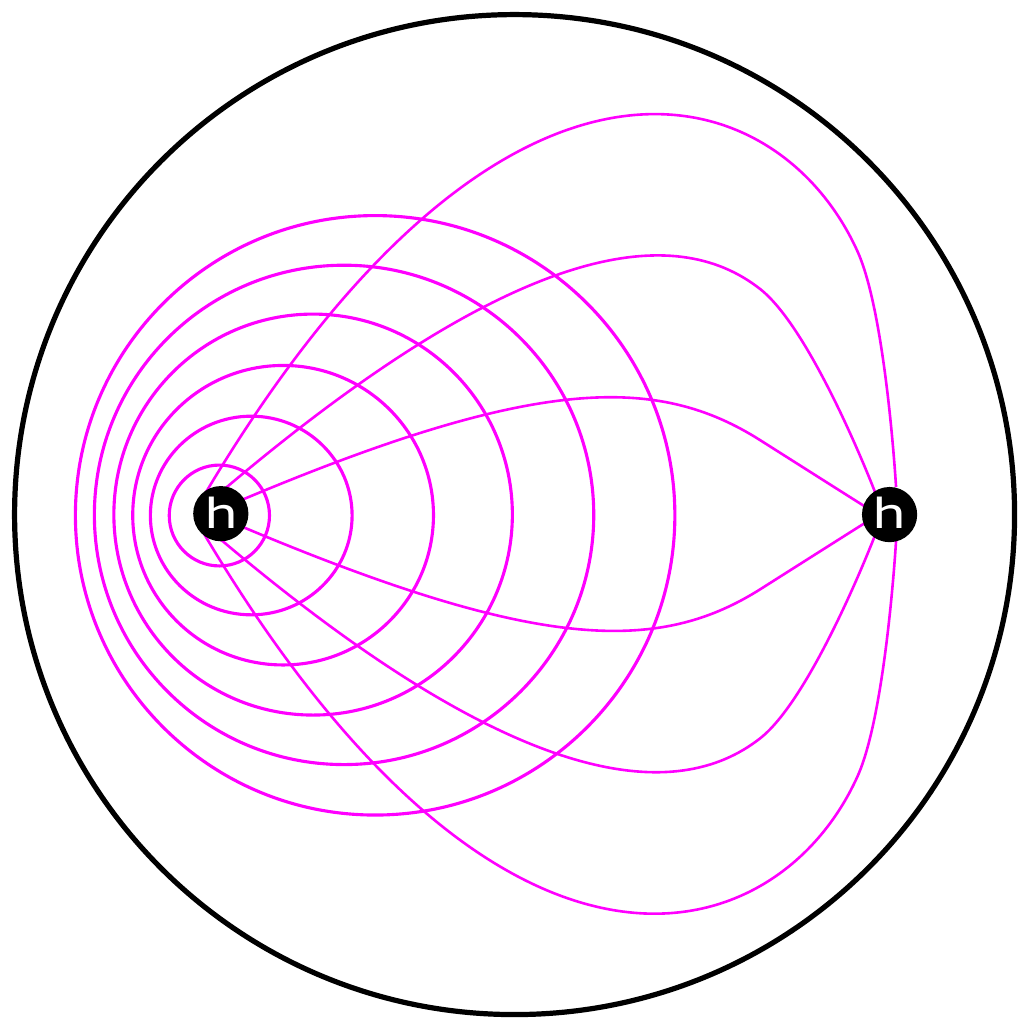}
	\caption{The vanishing cycles of the Lefschetz fibration $\pi_W : W \to B^2$.~\label{fig:6ab} }
\end{figure}

Now we turn our attention to $W,$ which is obtained by removing the interior of $V$ from the elliptic surface $E(1)$.  According to Lemma~\ref{lem: palf}, the standard elliptic fibration on $E(1)$ with $12$ singular fibers induces  a Lefschetz fibration  $\pi_W : W\to B^2$ whose vanishing cycles are depicted in Figure~\ref{fig:6ab}.   The corresponding $(13,2;1,1)$-relative trisection diagram for $W$ (see Remark~\ref{rem: order})  is depicted in Figure~\ref{fig:WTrisection}. 
The labeling of the surgeries from $1$ to $12$ in  Figure~\ref{fig:WTrisection} corresponds to the ordering of  the $12$ vanishing cycles $a_1,b_1,a_2, \ldots,  b_6$ of the Lefschetz fibration  $\pi_W : W\to B^2$ shown in Figure~\ref{fig:6ab}.

\begin{figure}
\centering
\begin{minipage}{.5\textwidth}
  \centering
  \includegraphics[width=.97\linewidth]{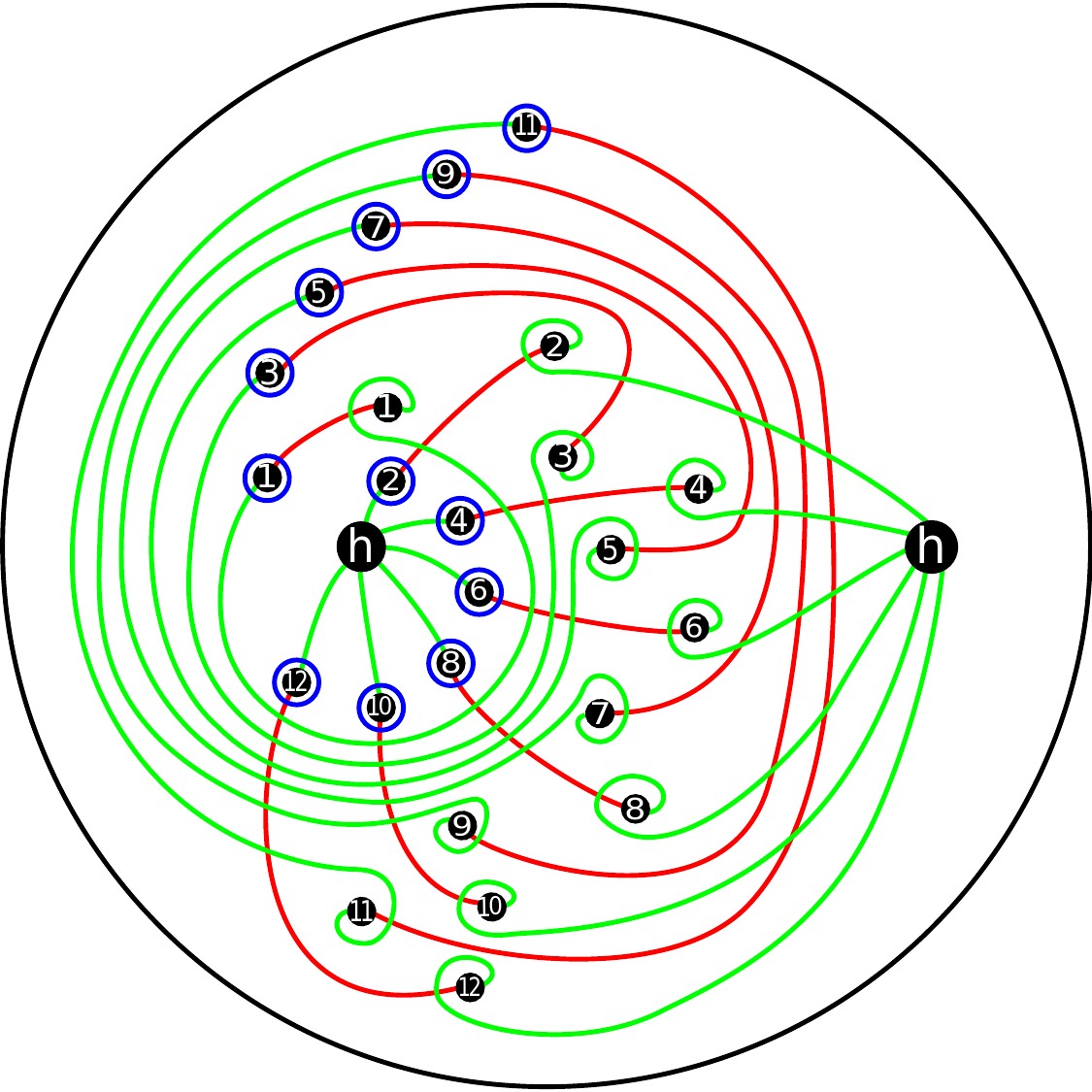}
  \captionof{figure}{The $(13,2;1,1)$-relative trisection diagram for $W$.}
 \label{fig:WTrisection}
\end{minipage}%
\begin{minipage}{.5\textwidth}
  \centering
  \includegraphics[width=.97\linewidth]{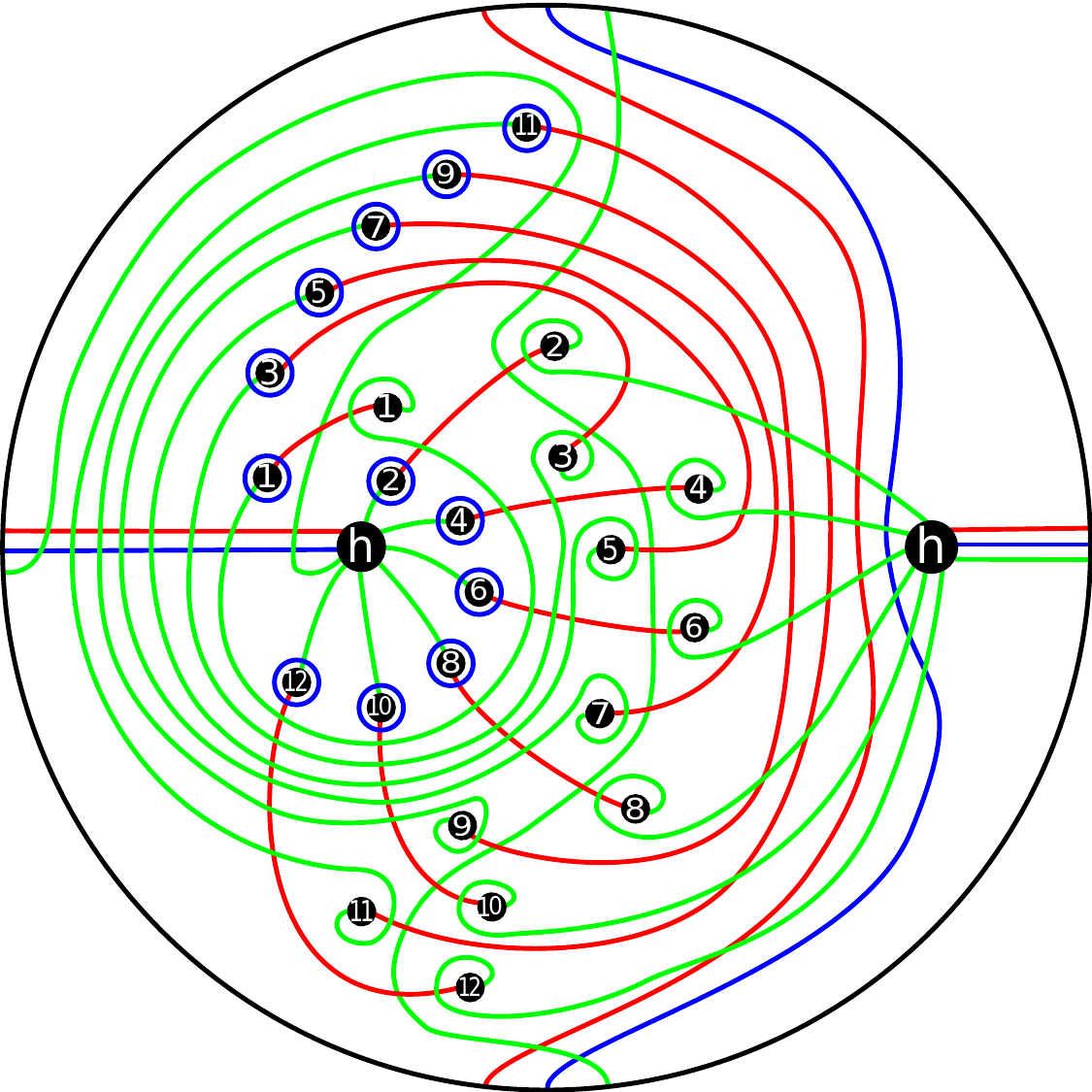}
  \captionof{figure}{The cut system  $(\A^W_\a, \A^W_\b, \A^W_\g)$ of arcs for the trisection diagram of $W$.}
  \label{fig:WArcs}
\end{minipage}
\end{figure}

We choose  a cut system  $(\A^W_\a, \A^W_\b, \A^W_\g)$ of arcs (see Figure~\ref{fig:WArcs}) for the relative trisection diagram for $W$ as follows.  We first choose  $\A^W_\a$ (the two red arcs) cutting $\Si_\a$ into a disk so that the end points of these two red arcs match with the end points of the two red arcs in $\A^V_\a$ in Figure~\ref{fig:VArcs}. Next,  we obtain the blue arcs in $\A^W_\b$ by taking parallel copies of  the red arcs without any need for handleslides. Finally,  we obtain the green arcs in $\A^W_\g$ by applying some handleslides to the parallel copies of the blue arcs.

 \begin{figure}[ht]
\includegraphics[scale=.55]{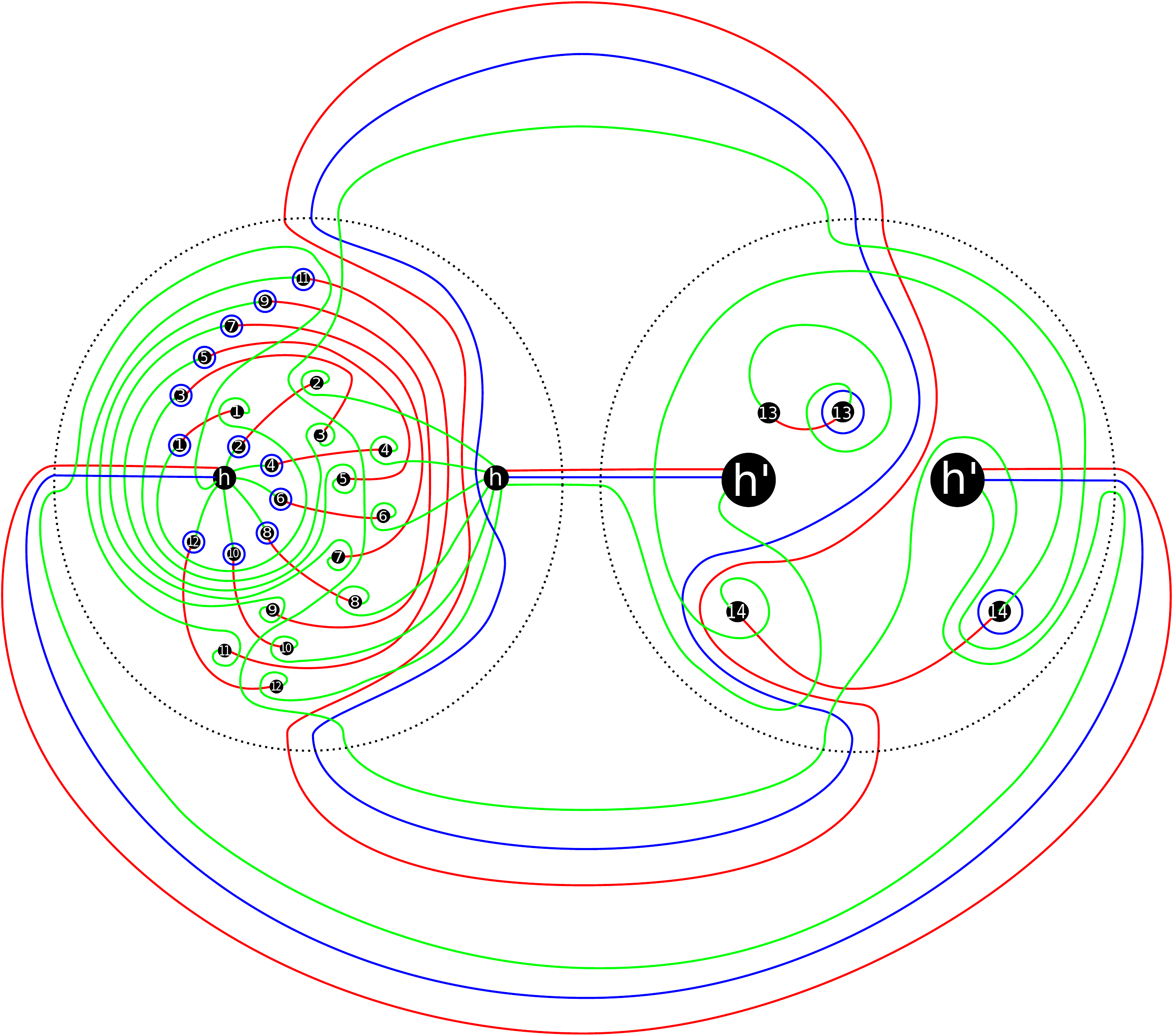}
\caption{A $(16, 2)$-trisection diagram of the elliptic surface $E(1)$.  There is copy of Figure~\ref{fig:WArcs} (resp. Figure~\ref{fig:VArcs}) inside the dotted circle on the left (resp. right), but the dotted circles are not part of the diagram.   ~\label{E1}}
\end{figure}
According to Proposition~\ref{prop:gluediagram}, in order  to obtain a $(16,2)$-trisection diagram for $E(1) = W \cup_{\del} V$, we  just have to glue the $(13,2;1,1)$-relative trisection diagram for $W$ decorated with  the cut system  $(\A^W_\a, \A^W_\b, \A^W_\g)$ of arcs depicted in Figure~\ref{fig:WArcs} with the  $(3,2;1,1)$-relative trisection diagram of $V$ decorated with  the cut system  $(\A^V_\a, \A^V_\b, \A^V_\g)$ of arcs  depicted in Figure~\ref{fig:VArcs} so that the end points of the red, blue and green arcs in Figure~\ref{fig:WArcs}  are identified with the end points of the the red, blue and green arcs in Figure~\ref{fig:VArcs}, respectively,  on the common boundary circle. We depicted a planar version of the  resulting $(16,2)$-trisection diagram for $E(1) = W \cup_{\del} V$ in Figure~\ref{E1}.


\begin{remark} Note that there is also a  genus $10$ trisection diagram of $E(1)$ viewed as the  connected sum $\cp \#9 \cpb$, which we  depicted in Figure~\ref{fig:9cp2}. The trisection diagrams of $E(1)$ shown in Figure~\ref{E1} and Figure~\ref{fig:9cp2} are related by stabilization, handleslides and diffeomorphism by \cite[Corollary  12]{gk}.  \end{remark}

\begin{figure}[ht] 
\includegraphics[scale=.8]{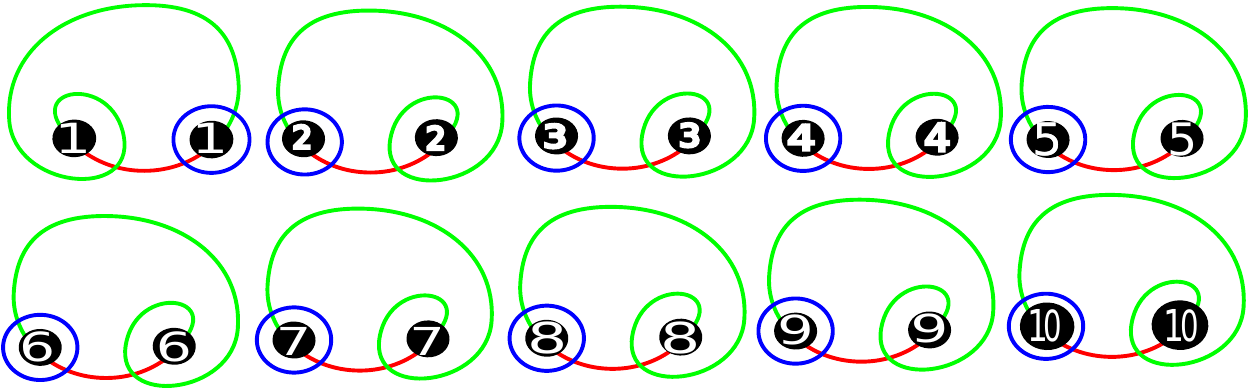}
\caption{A $(10,0)$-trisection diagram of $\cp \#9 \overline{\cp} \cong E(1)$.  ~\label{fig:9cp2}}
\end{figure}

The elliptic surface $E(n)$ for $n \geq 2$ would be the next natural candidate to trisect. Unfortunately, our method of proof of Theorem~\ref{thm: main}  does not {\em immediately} extend to cover  any of these $4$-manifolds. The elliptic fibration on $E(n)$ has a section of square $-n$. The regular neighborhood of the union of a non-singular fiber and the section of square $-n$ is a plumbing which admits an achiral Lefschetz fibration  as in Lemma~\ref{lem: plum}, but when $n \geq 2$,  the open book on the boundary does not match the open book coming from the Lefschetz fibration in the complement. 

On the other hand, $E(2) \#2 \cpb$ admits a genus $2$ Lefschetz fibration over $S^2$ with $30$ singular fibers, equipped with a section of square $-1$ (cf. \cite{m}). An implementation of Theorem~\ref{thm: main} gives a $(36, 4)$-trisection of $E(2) \#2 \cpb$. In the next subsection, we turn our attention to another well-known genus $2$ Lefschetz fibration on the Horikawa surface $H'(1)$.

\subsection{The Horikawa surface $H'(1)$} Our goal in this subsection is to present an explicit  $(46,4)$-trisection diagram of the Horikawa surface $H'(1)$, which is defined as the desingularization of the double branched cover of the Hirzebruch surface $\mathbb{F}_2$ (cf. \cite[page 269]{gs}). As we pointed out in the  introduction,   $H'(1)$ is a simply-connected complex surface of general type, which is homeomorphic but not diffeomorphic to $5 \cp \# 29 \cpb$.

 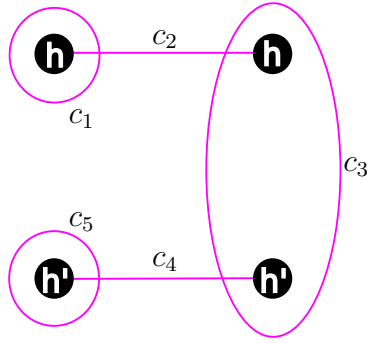
\begin{figure}
 \labellist
		\pinlabel $c_1$ at 100 280
		\pinlabel $c_2$ at 205 380
		\pinlabel $c_3$ at 445 220
		\pinlabel $c_4$ at 205 100
		\pinlabel $c_5$ at 100 150
		\endlabellist
\includegraphics[scale=.3]{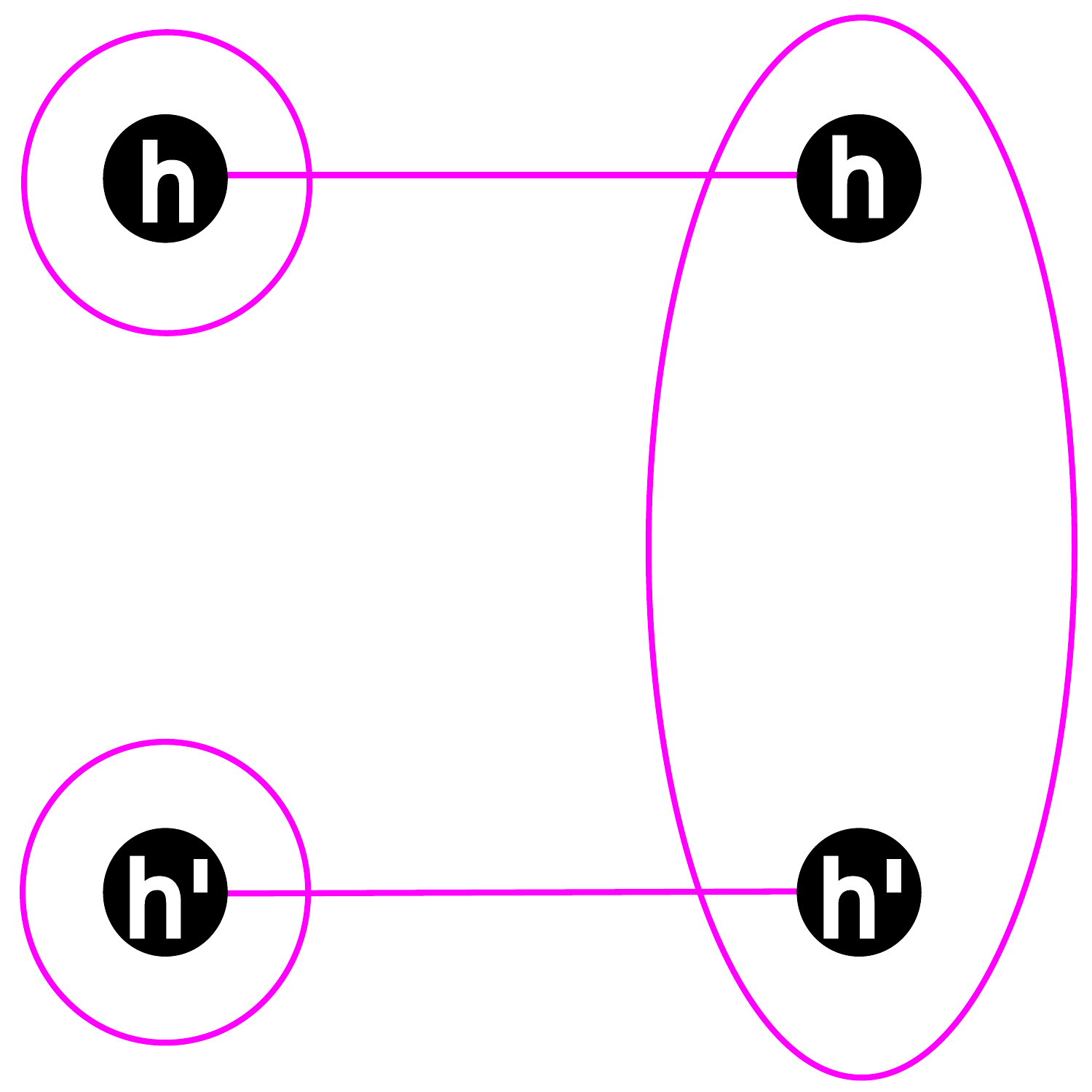}
\caption{The curves $c_1, \ldots, c_5$ on $\Si_2$. ~\label{fig:Genus2Planar}}
\end{figure}

The  relation $$1 = (D(c_1)D(c_2)D(c_3)D(c_4))^{10} \in  \G_{2}$$ where the curves $c_1,c_2,c_3,c_4$ are shown in Figure~\ref{fig:Genus2Planar},  defines a genus $2$ Lefschetz fibration over $S^2$. In \cite{f}, Fuller showed that the total space of this genus $2$ Lefschetz fibration is diffeomorphic to the Horikawa surface $H'(1)$. Moreover, it is easy to see that this fibration admits a section of square $-1$, since the relation above lifts to the chain relation $$D(\d)  = (D(c_1)D(c_2)D(c_3)D(c_4))^{10} \in  \G_{2,1}. $$  According to \cite{s}, $H'(1)$ admits a unique sphere of square $-1$, and therefore by blowing it down,  we also obtain a trisection diagram for the minimal simply connected complex surface $\widetilde{H'(1)}$ of general type,  by Theorem~\ref{thm: main}.  
 
\begin{figure}[ht] 
\includegraphics[scale=.3]{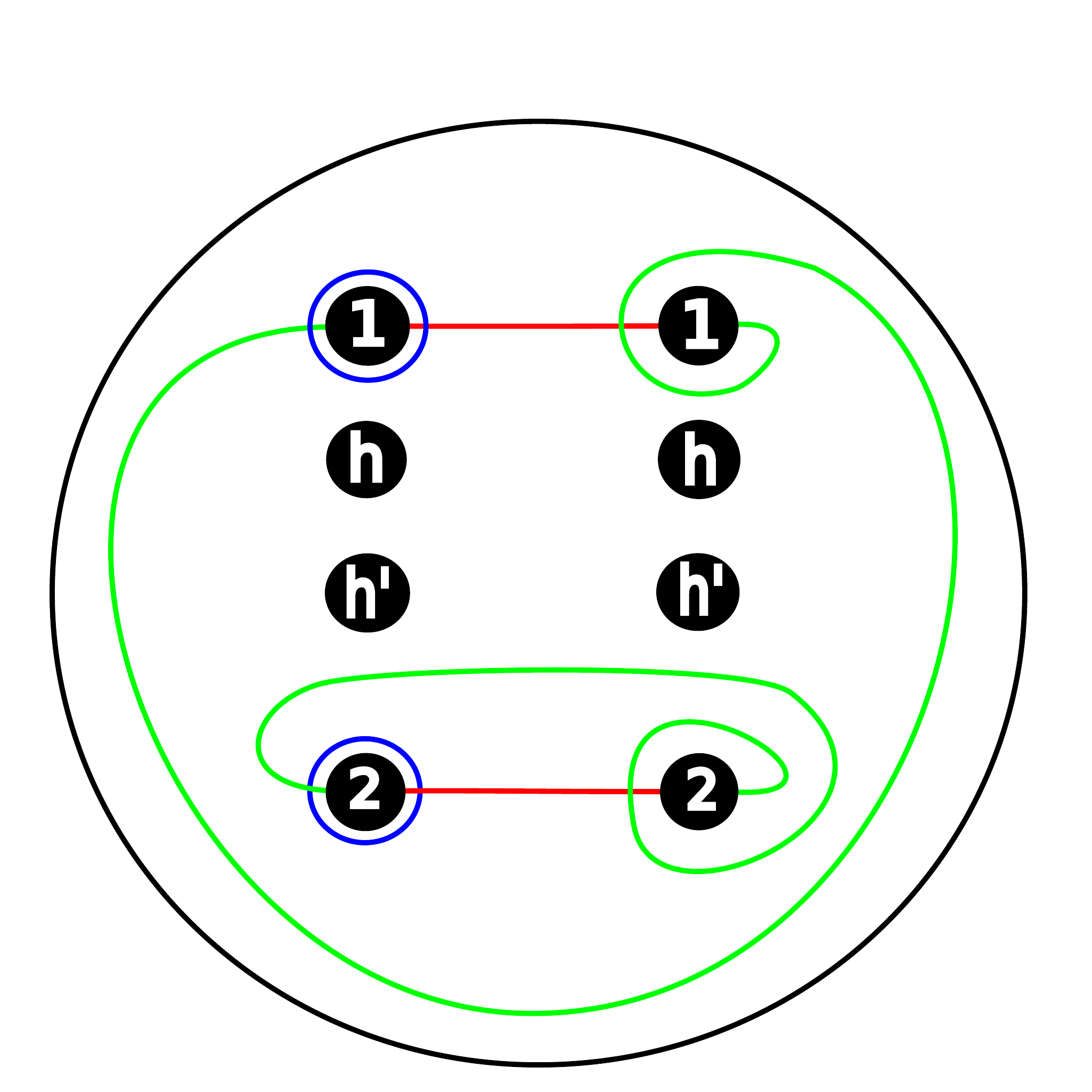}
\caption{The $(4,4;2,1)$-relative trisection diagram of $V$.  ~\label{fig:horikawa-V}}
\end{figure}

If $V$ denotes a regular neighborhood of the section union a nonsingular fiber, then we have  a decomposition of $H'(1)$ as $ W \cup_{\del} V$. By Lemma~\ref{lem: plum}, there is  an achiral Lefschetz fibration  $\pi_V : V \to B^2$ with regular fiber $\Si_{2,1}$, which has two vanishing cycles:  a homotopically  trivial curve  $\e$ with framing $-1$ and a boundary parallel curve $\d$ with framing $+1$.  By Remark~\ref{rem: order},   we obtain the corresponding $(4,4;2,1)$-relative trisection diagram of $V$ shown in Figure~\ref{fig:horikawa-V}, which is decorated  with  the cut system  $(\A^V_\a, \A^V_\b, \A^V_\g)$ of arcs  in Figure~\ref{fig:horikawa-Varcs}.

\begin{figure}[ht] 
\includegraphics[scale=.44]{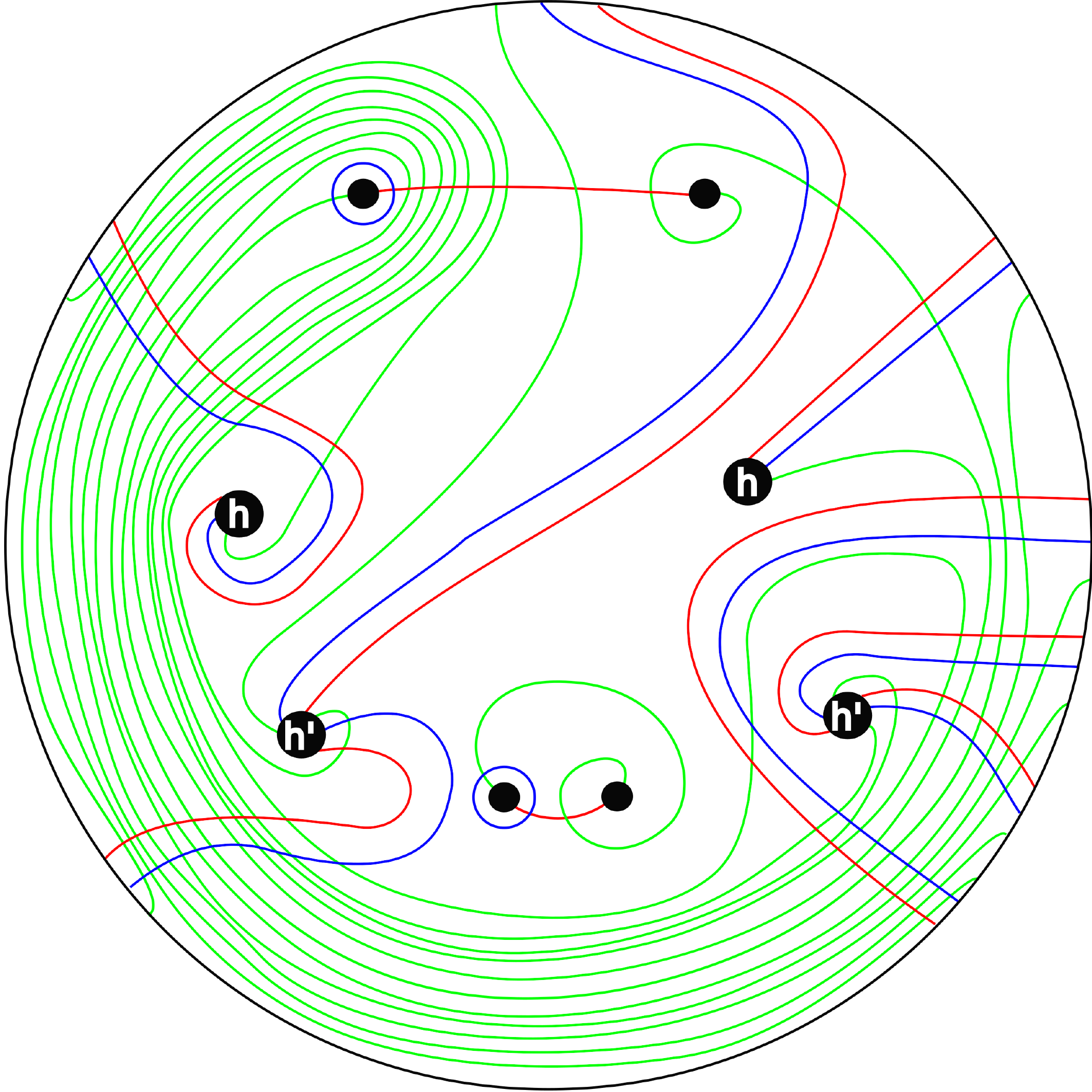}
\caption{The cut system  $(\A^V_\a, \A^V_\b, \A^V_\g)$ of arcs for $V$.  ~\label{fig:horikawa-Varcs}}
\end{figure}

On the other hand, we obtain the corresponding $(42,4;2,1)$-relative trisection diagram of $W$ shown in Figure~\ref{fig:horikawa-W} using the Lefschetz fibration $\pi_W: W \to B^2$ with monodromy $(D(c_1)D(c_2)D(c_3)D(c_4))^{10}$.  Note that we implemented Remark~\ref{rem: order} for the first $8$ vanishing cycles $c_4, c_3, c_2, c_1, c_4, c_3, c_2, c_1$ labeled by the surgeries $1$ through $8$ and the last $4$ vanishing cycles $c_4, c_3, c_2, c_1$ labeled by the surgeries $37$ through $40$.  We hope that the pattern is clear for the reader. In Figure~\ref{fig:horikawa-Warcs}, we depicted  the $(42,4;2,1)$-relative trisection diagram of $W$ decorated with  the cut system  $(\A^W_\a, \A^W_\b, \A^W_\g)$ of arcs.   

\begin{figure}[ht] 
\includegraphics[scale=.7]{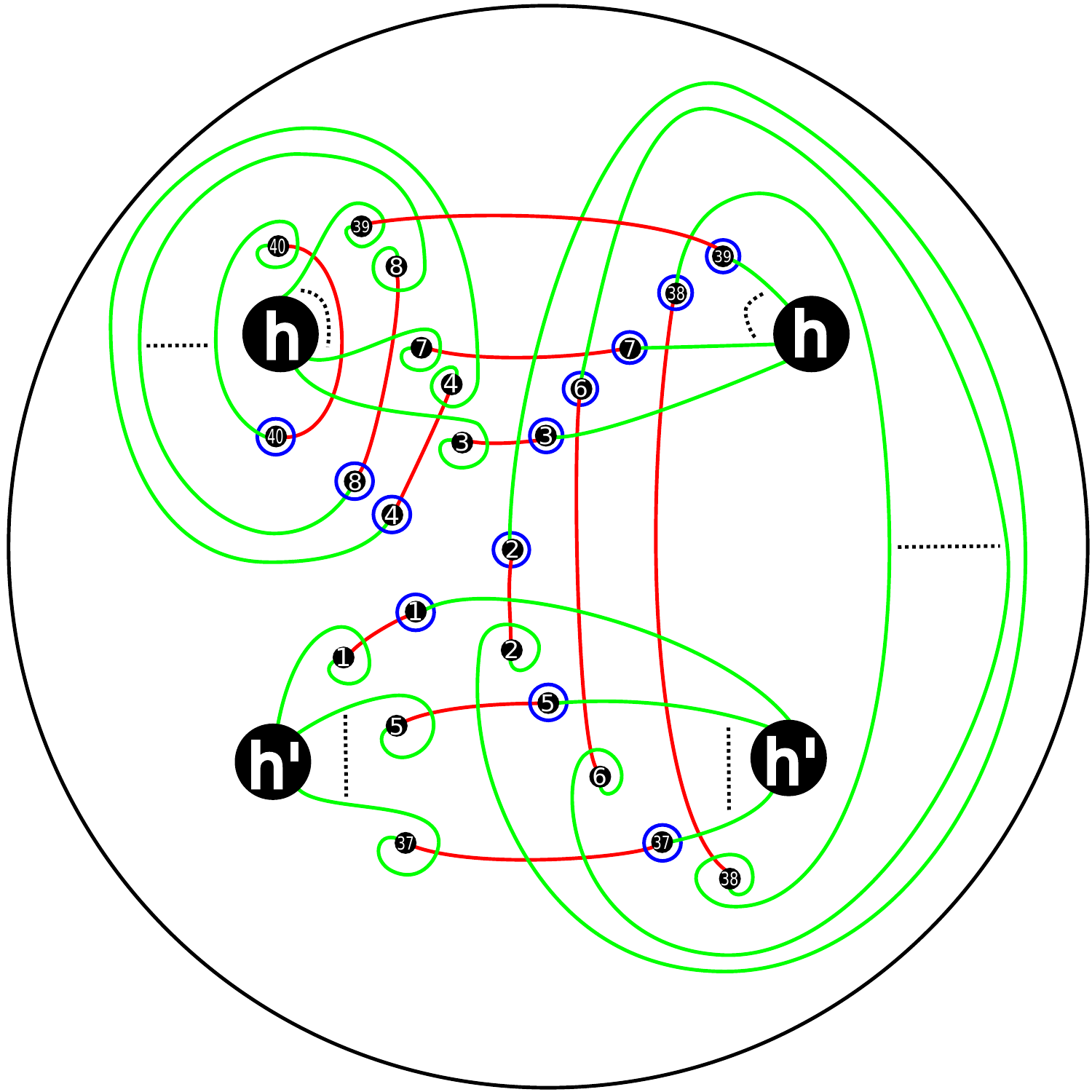}
\caption{The $(42,4;2,1)$-relative trisection diagram of $W$.   ~\label{fig:horikawa-W}}
\end{figure}

\begin{figure}[ht] 
\includegraphics[scale=.7]{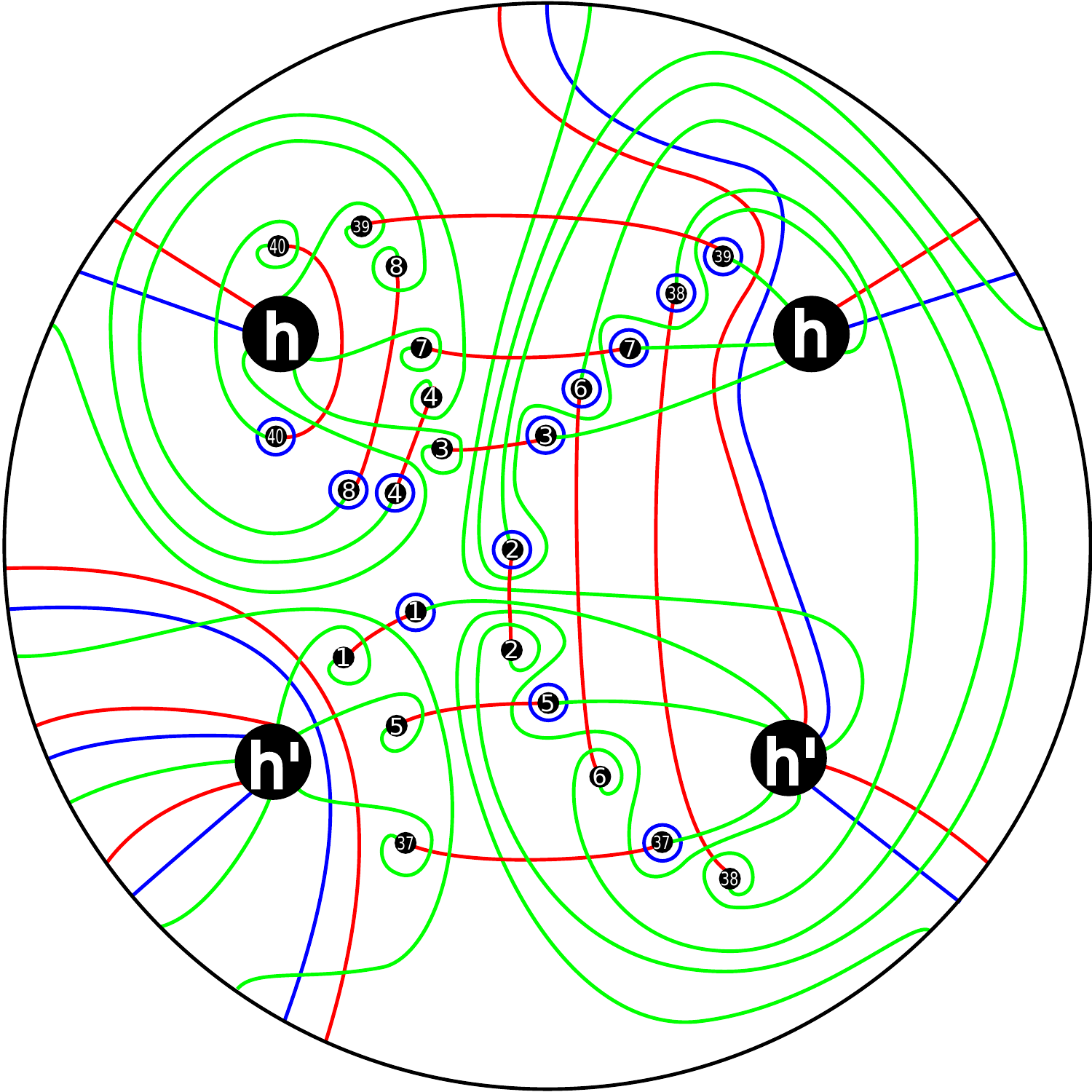}
\caption{The cut system  $(\A^W_\a, \A^W_\b, \A^W_\g)$ of arcs for $W$.  ~\label{fig:horikawa-Warcs}}
\end{figure}

Finally, in order  to obtain a $(46,4)$-trisection diagram for $H'(1) = W \cup_{\del} V$ using Proposition~\ref{prop:gluediagram}, we  just have to glue the $(42,4;2,1)$-relative trisection diagram for $W$ decorated with  the cut system  $(\A^W_\a, \A^W_\b, \A^W_\g)$ of arcs depicted in Figure~\ref{fig:horikawa-Warcs} with the  $(4,4;2,1)$-relative trisection diagram of $V$ decorated with  the cut system  $(\A^V_\a, \A^V_\b, \A^V_\g)$ of arcs  depicted in Figure~\ref{fig:horikawa-Varcs} so that the end points of $(\A^W_\a, \A^W_\b, \A^W_\g)$ are identified with the end points of $(\A^V_\a, \A^V_\b, \A^V_\g)$ on the common boundary circle. The end result is depicted in Figure~\ref{fig:horikawa-glued}.

\begin{figure}[ht] 
\includegraphics[scale=.5]{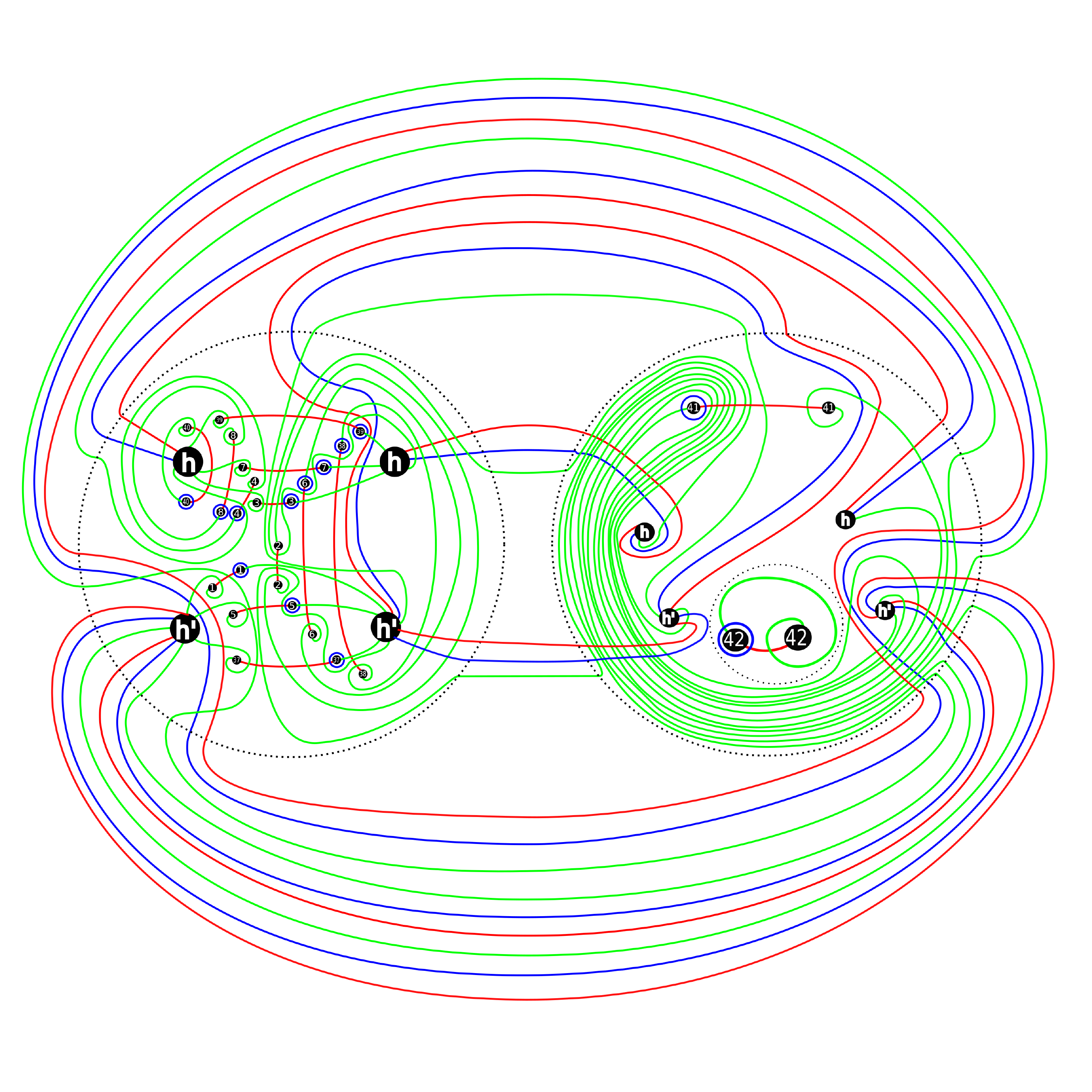}
\caption{ A $(46,4)$-trisection diagram for the Horikawa surface $H'(1)$.  By erasing the surgery labelled by $42$ and the three curves associated to that in the small dotted circle,  we obtain a $(45,4)$-trisection diagram for the minimal simply connected complex surface $\widetilde{H'(1)}$ of general type. ~\label{fig:horikawa-glued}}
\end{figure}

\begin{remark} Similar to the diagram depicted in Figure~\ref{fig:9cp2}, there is indeed a  $(34,0)$-trisection diagram of  $5 \cp \# 29 \cpb$  which can be  
stabilized four  times to yield a  $(46,4)$-trisection diagram.  As a consequence, we obtain explicit $(46,4)$-trisection diagrams for a pair of closed $4$-manifolds, the Horikawa surface $H'(1)$ and  $5 \cp \# 29 \cpb$, which are homeomorphic but not diffeomorphic. 
\end{remark}

\subsection{Trivial $S^2$ bundle over $\Si_h$} \label{subsec: trivial} In \cite{gk}, there is a description of a $(8h+5, 4h+1)$-trisection for any oriented $\Si_h$-bundle over $S^2$, including of course $\Si_h \times S^2$, without an explicit diagram. In Example~\ref{ex: eng}, however, we obtained a $(2h+5, 2h+1)$-trisection of $\Si_h \times S^2$. Here we illustrate our doubling technique by drawing a $(7,3)$-trisection diagram for $T^2 \times S^2$. 

We first observe that $T^2 \times B^2$ admits an achiral Lefschetz fibration over $B^2$ with fiber  $\Si_{1,2}$, which carries  two vanishing cycles $\d_1$ and $\d_2$, each of which is parallel to one boundary component of $\Si_{1,2}$. The monodromy of the this fibration is given by $D(\d_1)D^{-1}(\d_2)$. Implementing  Remark~\ref{rem: order}, we obtain the $(3,3;1,2)$-relative trisection diagram for $T^2 \times B^2$ shown in Figure~\ref{fig:T2XB2} (see also  \cite[Figure 15(b)]{cgp}). A cut system of arcs for this trisection diagram of $T^2 \times B^2$ is depicted in Figure~\ref{fig:T2XB2_WithArcs}. 

To double the $(3,3;1,2)$-relative trisection diagram decorated with the cut system of arcs in  Figure~\ref{fig:T2XB2_WithArcs}, we simply draw its mirror image (for the orientation reversal)  next to it and identify  the inner and outer boundary components, respectively,  to each other. It is clear how to identify the outer boundary components and glue the arcs with the same color, as we illustrated many times so far in this paper. To identify the inner boundary components and still draw a planar diagram, we replace the inner boundaries by the surgery labeled by $5$ in Figure~\ref{fig:T2xS2}. As a result we obtain the $(7,3)$-trisection diagram for $T^2 \times S^2$ as shown in Figure~\ref{fig:T2xS2}, where we isotoped some of the curves after implementing Corollary~\ref{cor: diag}.  One can similarly draw a  $(2h+5, 2h+1)$-trisection diagram of  $\Si_h \times S^2$ for any $ h \geq 2$.

\begin{figure}
\centering
\begin{minipage}{.5\textwidth}
  \centering
  \includegraphics[width=.9\linewidth]{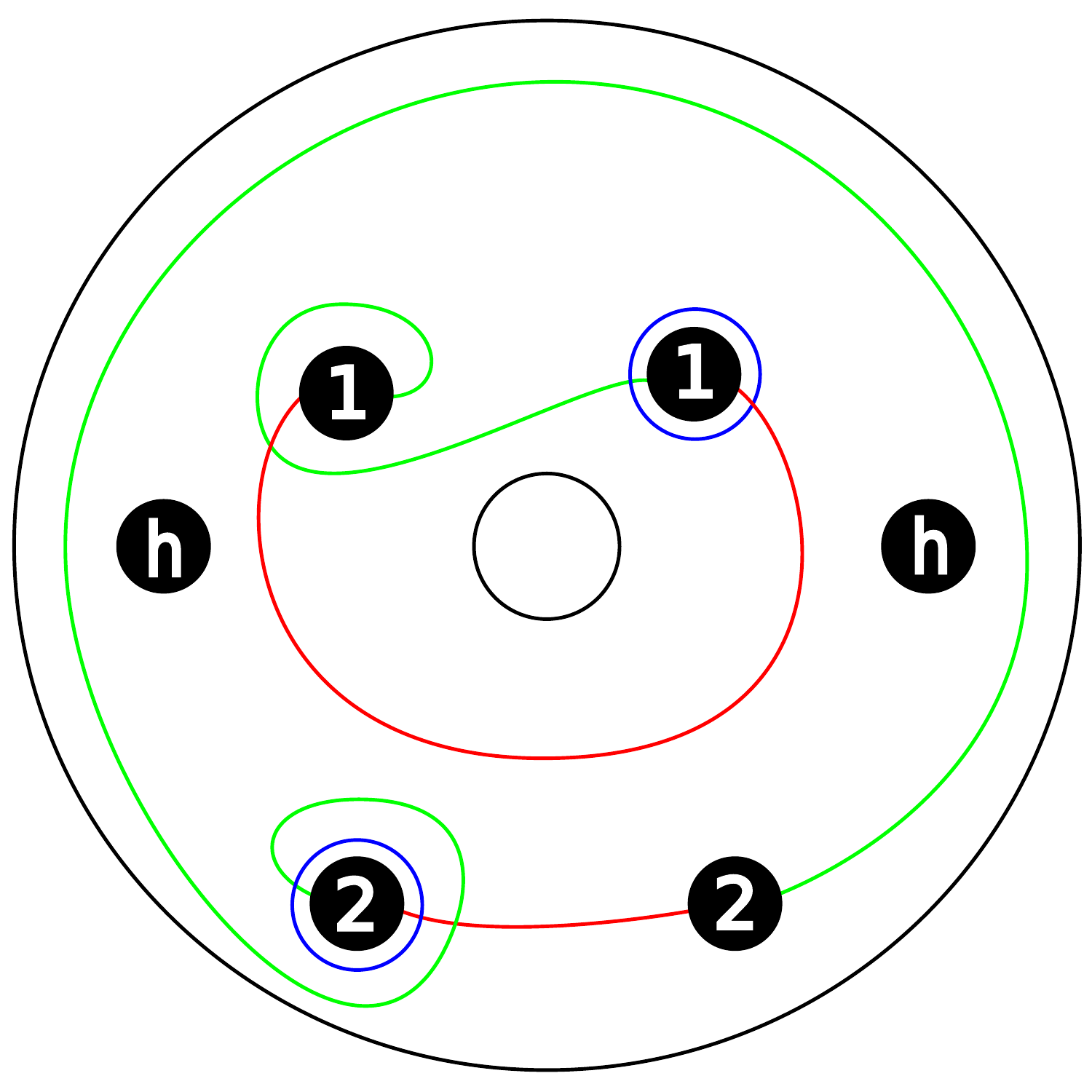}
  \captionof{figure}{The $(3,3;1,2)$-relative trisection diagram for $T^2 \times B^2$.}
  \label{fig:T2XB2}
\end{minipage}%
\begin{minipage}{.5\textwidth}
  \centering
  \includegraphics[width=.9\linewidth]{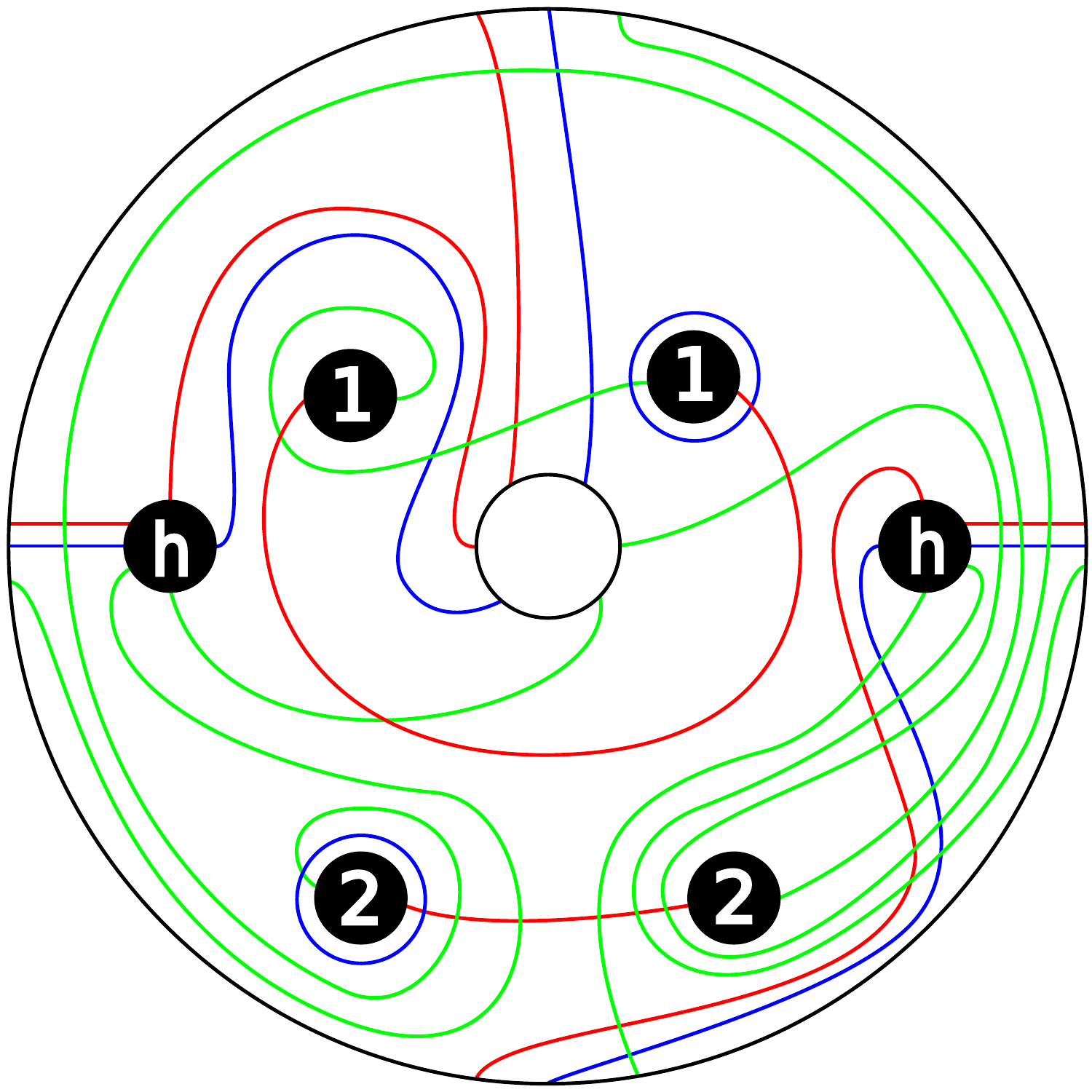}
  \captionof{figure}{The cut system of arcs for  the relative trisection diagram of $T^2 \times B^2$.}
  \label{fig:T2XB2_WithArcs}
\end{minipage}
\end{figure}

\begin{figure}[ht] 
\includegraphics[scale=.6]{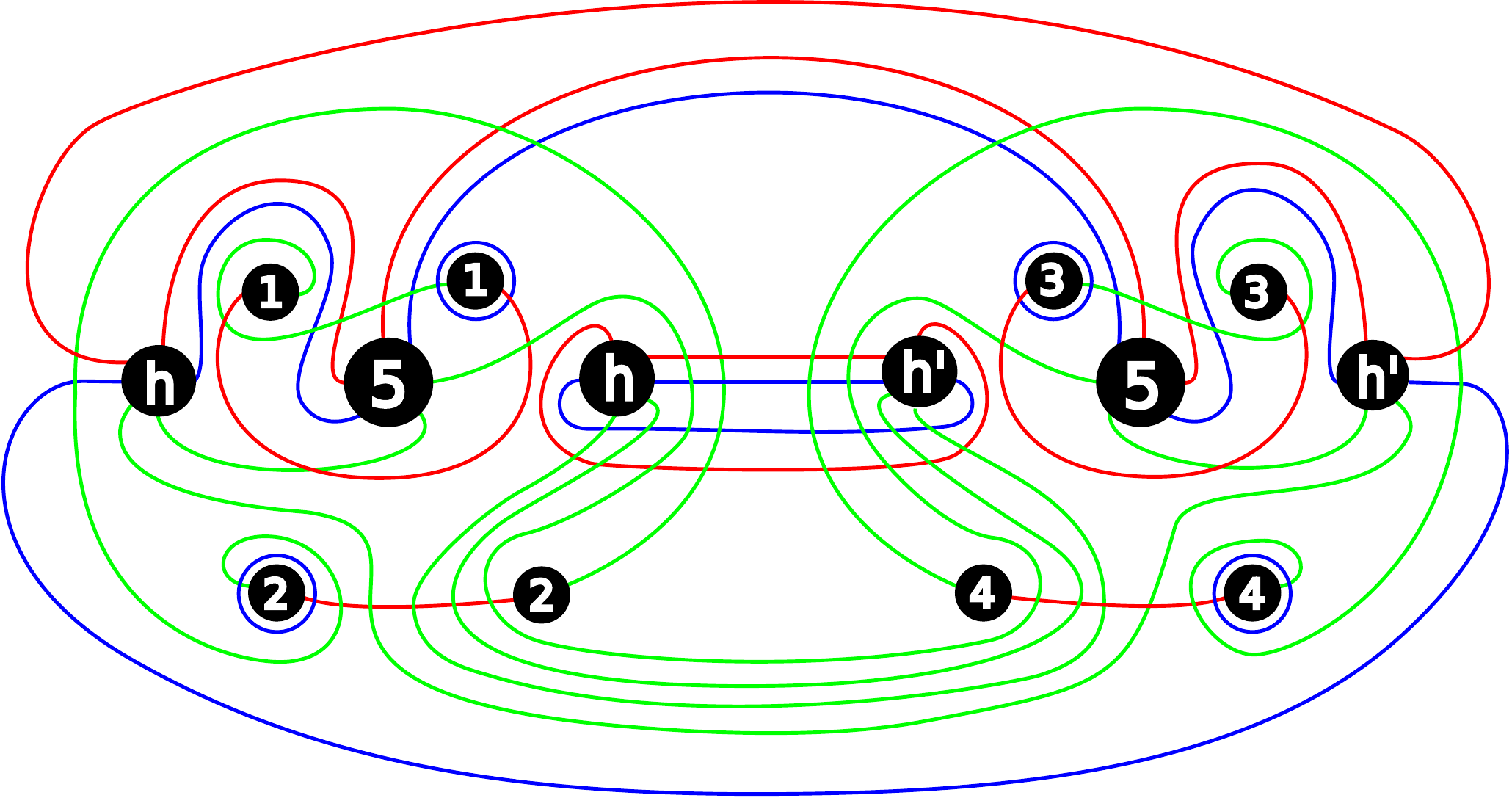}
\caption{A $(7,3)$-trisection diagram for $T^2 \times S^2$.    ~\label{fig:T2xS2}}
\end{figure}

\subsection{Twisted $S^2$-bundle over $\Si_h$} \label{subsec: twisted}  In Example~\ref{ex: eng}, we obtained a $(2h+2, 2h)$-trisection of the non-trivial $S^2$-bundle $ \Si_h \tilde{\times} S^2$. Except for $S^2 \tilde{\times} S^2 \cong \cp \# \cpb$, which has a $(2,0)$-trisection,  these bundles are not covered by the examples in \cite{gk}.

In the following, we illustrate our doubling technique described in Corollary~\ref{cor: diag}  by drawing a $(4,2)$-trisection diagram for $T^2 \tilde{\times} S^2$.  Note that $E_{1,1}$, the $B^2$-bundle over $T^2$ with Euler number $+1$, admits an achiral Lefschetz fibration over $B^2$ whose regular fiber is $\Si_{1,1}$ and which has only one singular fiber whose vanishing cycle is the  boundary parallel curve $\d \subset \Si_{1,1}$.  Therefore, by Lemma~\ref{lem: leftri}, $E_{1,1}$ has a $(2,2;1,1)$-relative trisection, whose diagram, decorated with the cut system of arcs, is depicted  in Figure~\ref{fig:E_11WithArcs}.

\begin{figure}[ht] 
\includegraphics[scale=.5]{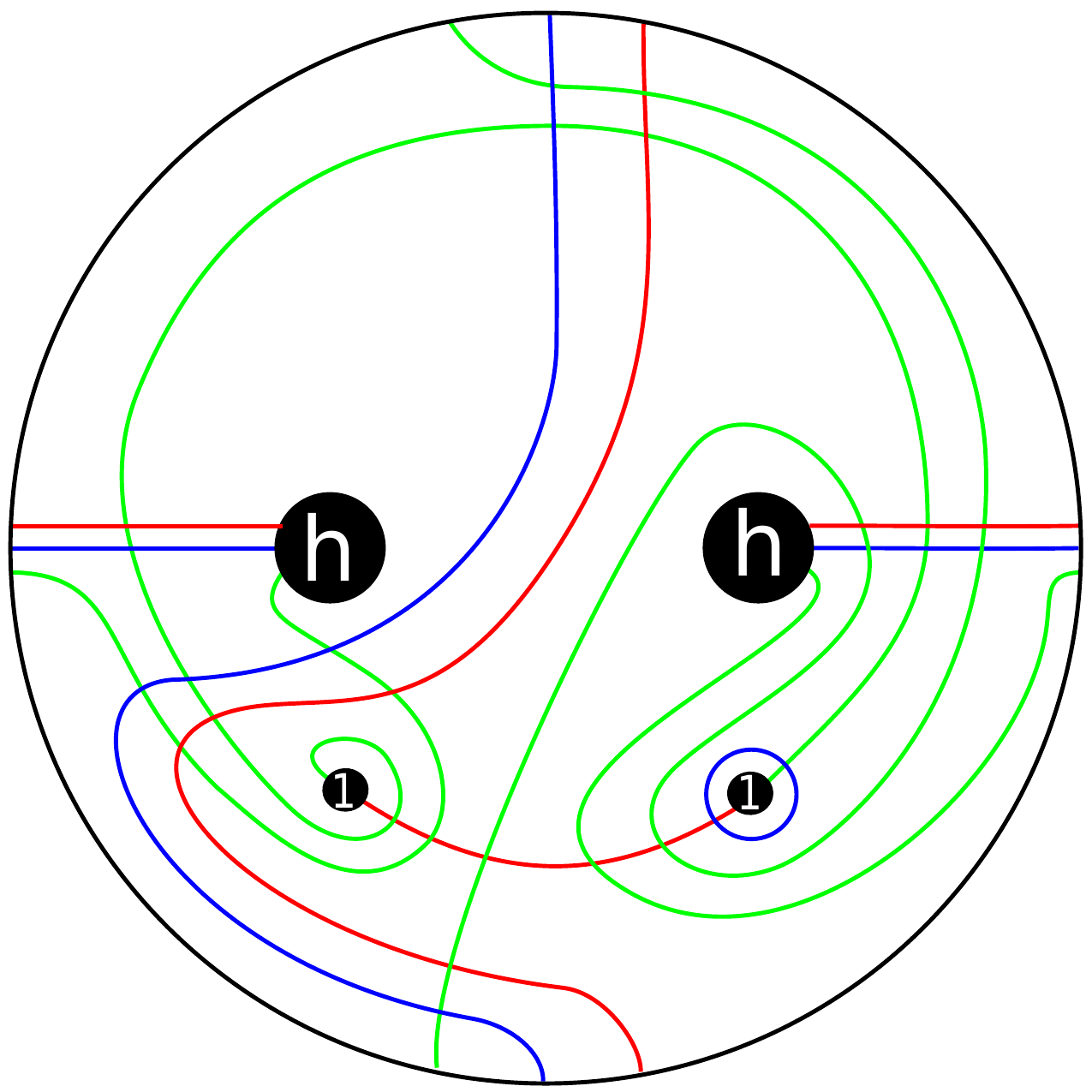}
\caption{The cut system of arcs for the $(2,2;1,1)$-relative trisection diagram of the $B^2$-bundle over $T^2$ with Euler number $+1$. ~\label{fig:E_11WithArcs}}
\end{figure}

\begin{figure}[ht] 
\includegraphics[scale=.8]{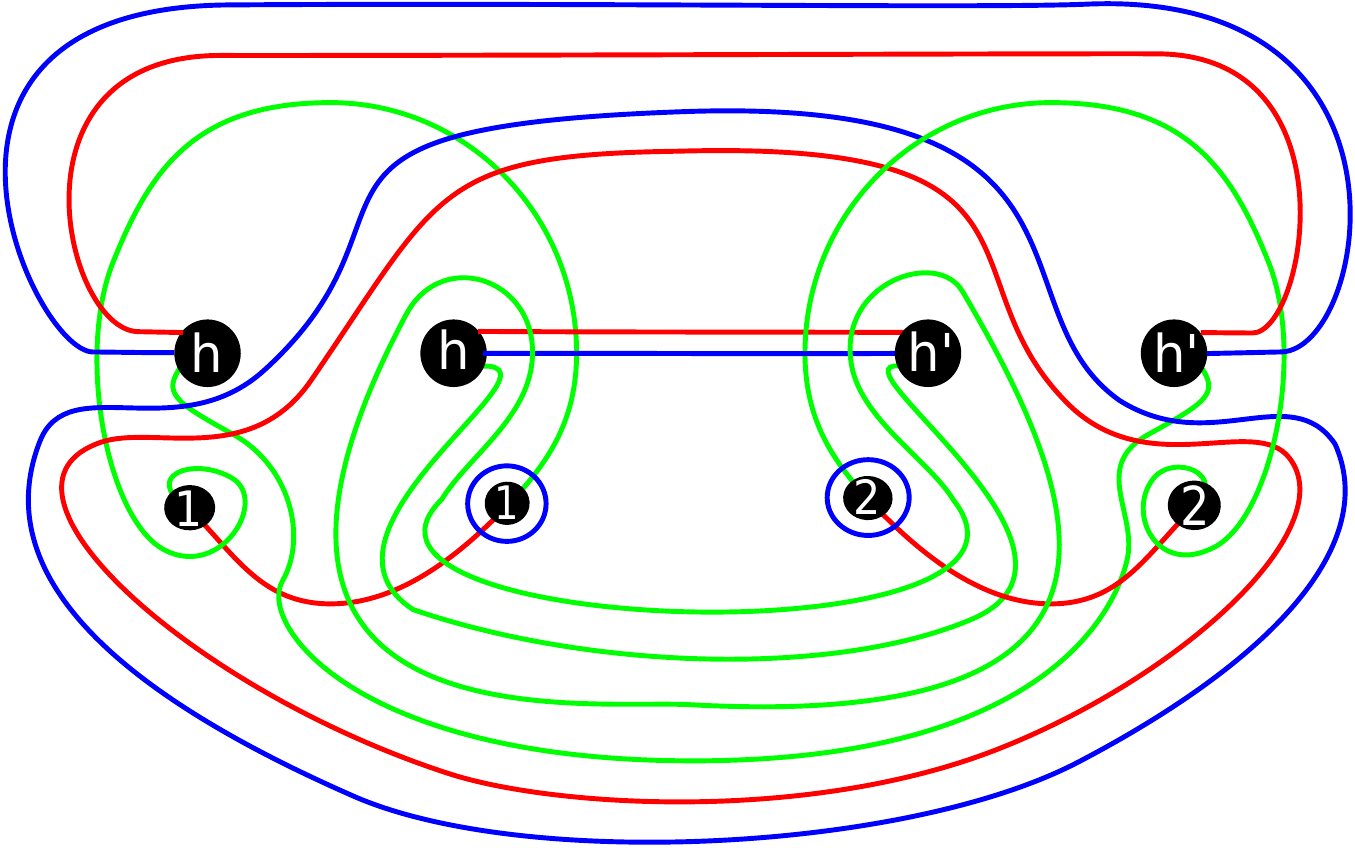}
\caption{A $(4,2)$-trisection diagram for $T^2 \tilde{\times} S^2$.   ~\label{fig:Twisted_S2xT2}}
\end{figure}

Since the double of  $E_{1,1}$ is  $T^2 \tilde{\times} S^2$, we obtain the  $(4,2)$-trisection diagram for  $T^2 \tilde{\times} S^2$ in Figure~\ref{fig:Twisted_S2xT2},  by doubling the diagram in Figure~\ref{fig:E_11WithArcs}.  This is done by drawing  Figure~\ref{fig:E_11WithArcs} and its mirror image next to it and gluing the arcs of the same color.   

  One can similarly draw a  trisection diagram of $ \Si_h \tilde{\times} S^2$  for any $ h \geq 2$. Note that  for $h=0$ our doubling technique gives the standard  $(2,0)$-trisection diagram for $\cp \# \cpb$, which is indeed diffeomorphic to $S^2 \tilde{\times} S^2$.

\subsection{Oriented $S^2$-bundles over non-orientable surfaces} \label{subsec:nonorie} There are two oriented $S^2$-bundles over any closed connected  non-orientable surface $N_h = \#^h \mathbb{RP}^2$ of genus $h$, both of which are obtained as doubles of $B^2$-bundles over   $N_h$. The disk cotangent bundle $DT^*N_h$ of $N_h$, which is diffeomorphic to the $B^2$-bundle over $N_h$ with Euler number $h-2$, admits a Lefschetz fibration over $B^2$ whose regular fiber is $\Si_{0, 2h+2}$  and which has $h+2$ singular fibers (cf. \cite{oo}). According to Lemma~\ref{lem: leftri},  $DT^*N_h$ has a $(h+2, 2h+1, 0, 2h+2)$-relative trisection, and one can draw the corresponding diagram by Remark~\ref{rem: order}. 

It follows that for any $h \geq 1$, one of the oriented $S^2$-bundles over $N_h$ has a $(4h+5, 2h+1)$-trisection by Corollary~\ref{cor: diag}, since it is the double of $DT^*N_h$.   
In particular,  a $(3,3,0,4)$-relative trisection diagram of the disk cotangent bundle of $\mathbb{RP}^2$,  a rational homology ball with boundary the lens space $L(4,1)$, is depicted in \cite{cgp}.  The  double of this rational homology ball is  an oriented $S^2$-bundle over $\mathbb{RP}^2$, which has a $(9,3)$-trisection whose corresponding diagram can be obtained by doubling the $(3,3,0,4)$-relative trisection diagram  as illustrated in the previous sections.
\section{Alternate proof of the Gay-Kirby Theorem} 

In this section, we provide an alternate proof of Theorem~\ref{thm: existence}. We refer to \cite{et, gs, os} for the definitions and properties of Lefschetz fibrations, open books, contact structures, etc.

\begin{theorem} [Gay and Kirby \cite{gk}] \label{thm: existence} Every smooth, closed, oriented, connected $4$-manifold admits a trisection.  \end{theorem}  
\begin{proof} [Alternate proof of Theorem~\ref{thm: existence}] 

Suppose that $X$ is a closed $4$-manifold. According to Etnyre and Fuller \cite[Proposition 5.1]{ef}, there is an embedded $3$-manifold $M \subset X$ satisfying the following properties: 

\begin{itemize} 
  \item There exists a decomposition $X= W  \cup_M W'$, where $\del W = M = -\del W'$. 
  \item Each of  $W$ and $W'$  admits an achiral Lefschetz fibration over $B^2$ with bounded fibers. 
  \item The two open books induced by the respective achiral Lefschetz fibrations on $W$ and $W'$ coincide on $M$. 
  \end{itemize}

As we described in Lemma~\ref{lem: leftri},  there is a straightforward method to turn the total space of an achiral Lefschetz fibration over $B^2$ into a relative trisection so that the respective open books induced by the relative trisection and the Lefschetz fibration agree on the boundary. 

Consequently, we have a decomposition of $X$ into two pieces $W$ and $W'$,  each of which has an explicit relative trisection so that the induced open books on the boundary coincide with the appropriate orientations.  We rely on Lemma~\ref{lem:gluem}  to finish the proof of Theorem~\ref{thm: existence}.  \end{proof}

\begin{remark} \label{rem: con} Here we briefly outline the proof of the aforementioned result of Etnyre and Fuller to indicate how contact geometry enters the scene.  Suppose that the $4$-manifold $X$ is given by a handle decomposition. We first realize the union of $0$-, $1$-, and $2$-handles  as an achiral Lefschetz fibration over $B^2$, whose vanishing cycles can be explicitly  described using the technique in \cite{h}. Similarly, the union of $3$- and $4$-handles also admits a Lefschetz fibration over $B^2$. But there is indeed no reason for the open books on the boundary to coincide at this point. However, by stabilizing both achiral Lefschetz fibrations several times, the contact structures supported by the resulting open books on the boundary become homotopic as oriented plane fields. Moreover, both contact structures may be assumed to be overtwisted by negatively stabilizing, if necessary. Therefore, we conclude that the contact structures are isotopic by Eliashberg's classification \cite{e}.  We achieve the desired decomposition of $X$,  using a fundamental result of Giroux \cite{g}, which says that if two open books support isotopic contact structures on a closed $3$-manifold, then they have a common (positive) stabilization.  \end{remark}

\begin{remark} Suppose that we have a splitting of a closed $4$-manifold $X$ into two achiral Lefschetz fibrations over $B^2$ --- not necessarily by the method of Etnyre and Fuller as outlined above --- inducing the same open book on their common  boundary.  Then the next two steps in the proof of Theorem~\ref{thm: existence} provide a trisection of $X$. On the other hand, a  further positive or negative stabilization of  the common open book on the boundary yields  a stabilization of each of the Lefschetz fibrations  in the splitting of $X$, which in turn,  induces a new (higher genus) trisection of $X$.  This new trisection of $X$ can be obtained by  a stabilization of  the initial trisection.  \end{remark} 

\begin{remark} The decomposition theorem of Etnyre and Fuller that we quoted at the beginning of our proof of Theorem~\ref{thm: existence}  was improved by Baykur \cite[Theorem 5.2]{b} who showed that the Lefschetz fibration on $W$ (resp. $W'$) can be assumed to have only positively (resp. negatively) oriented Lefschetz singularities.  \end{remark}

\noindent{\bf{Acknowledgement:}} This work was initiated at the ``Trisections and low-dimensional topology" workshop at the American Institute of Mathematics (AIM). We would like express our gratitude to AIM for its hospitality.  We would also like thank D. T. Gay and R. \.{I}. Baykur for their helpful comments on a draft of this paper.

\end{document}